\documentclass[10pt]{amsart}
\theoremstyle{plain}
\usepackage{amsmath, amsfonts, amsthm, amscd, mathrsfs}
\usepackage{graphics}
\usepackage{enumerate}
\usepackage{color}
\usepackage{epsfig}
\usepackage[all]{xy}
\usepackage{graphicx,amssymb}

\graphicspath{ {Figures/} }
\numberwithin{figure}{section}

\theoremstyle{theorem}
\newtheorem*{thm}{Theorem}

\newtheorem{theorem}{Theorem}[section]

\newtheorem{cor}[theorem]{Corollary}
\newtheorem{construction}[theorem]{Construction}
\newtheorem{lemma}[theorem]{Lemma}
\newtheorem{step}{Step}
\newtheorem{prop}[theorem]{Proposition}

\newcommand{\G}{\PSL_2 \mathbb{C}}
\newcommand{\g}{\Gamma}
\newcommand{\ie}{{\itshape i.e.} }
\renewcommand{\to}[1][]{\xrightarrow{\ #1\ }}

\newcommand{\bR}{\mathbb{R}}

\newcommand{\bC}{\mathbb{C}}

\newcommand{\bT}{\mathbb{T}}
\newcommand{\bZ}{\mathbb{Z}}

\newcommand{\bH}{\mathbb{H}}
\newcommand{\cC}{\mathcal{C}}
\newcommand{\cE}{\mathcal{E}}
\newcommand{\cM}{\mathcal{M}}

\newcommand{\cP}{\mathcal{P}}
\newcommand{\cU}{\mathcal{U}}
\newcommand{\cK}{\mathcal{K}}
\newcommand{\cS}{\mathcal{S}}
\newcommand{\cT}{\mathcal{T}}
\newcommand{\cH}{\mathcal{H}}
\newcommand{\cEL}{\mathcal{EL}}
\newcommand{\teich}{\mathscr{T}}
\newcommand{\sS}{\mathscr{S}}

\newcommand{\inverse}{^{-1}}

\DeclareMathOperator{\cl}{cl }

\DeclareMathOperator{\hyp}{hyp}
\DeclareMathOperator{\collar}{collar}

\DeclareMathOperator{\Mod}{Mod}
\DeclareMathOperator{\Hyp}{Hyp}
\DeclareMathOperator{\Hb}{H}
\DeclareMathOperator{\Cb}{C}

\DeclareMathOperator{\PSL}{PSL}
\DeclareMathOperator{\mass}{mass}
\DeclareMathOperator{\vol}{Vol}
\DeclareMathOperator{\area}{Area}
\DeclareMathOperator{\im}{im}

\DeclareMathOperator{\str}{str}

\DeclareMathOperator{\dR}{dR}
\DeclareMathOperator{\interior}{int}

\DeclareMathOperator{\base}{base}

\begin{document}
\title{Bounded cohomology of finitely generated Kleinian groups}
\author{James Farre}
\maketitle

\begin{abstract} Any action of a group $\Gamma$ on $\bH^3$ by isometries yields a class in degree three bounded cohomology by pulling back the volume cocycle to $\Gamma$.  We prove that the bounded cohomology of finitely generated Kleinian groups without parabolic elements distinguishes the asymptotic geometry of geometrically infinite ends of hyperbolic $3$-manifolds.  That is, if two homotopy equivalent hyperbolic manifolds with infinite volume and without parabolic cusps have different geometrically infinite end invariants, then they define a $2$ dimensional subspace of bounded cohomology.   Our techniques apply to classes of hyperbolic $3$-manifolds that have sufficiently different end invariants, and we give explicit bases for vector subspaces whose dimension is uncountable.  We also show that these bases are uniformly separated in pseudo-norm, extending results of Soma.  The technical machinery of the Ending Lamination Theorem allows us to analyze the geometrically infinite ends of hyperbolic $3$-manifolds with unbounded geometry. \end{abstract}

\section{Introduction}

The bounded cohomology of groups and spaces behaves very differently from the ordinary cohomology.   In degree one with trivial coefficients, the bounded cohomology of any discrete group vanishes, while the bounded cohomology  of a non-abelian free group $F_2$ in degree two is a Banach space with dimension the cardinality of the continuum (see for example \cite{brooks:remarks} and \cite{mm:banach}).  In this paper, we study the degree three bounded cohomology of finitely generated Kleinian groups with infinite co-volume.  Computing bounded cohomology has remained elusive more than 35 years after Gromov's seminal paper \cite{gromov:bdd}.  For example, it is not known if the bounded cohomology of non-abelian free groups vanishes in degree $4$ and higher.  

To find non-trivial classes in the degree three bounded cohomology, we study the asymptotic geometry of geometrically infinite ends of hyperbolic $3$-manifolds.  
In manifolds without parabolic cusps, these ends are parameterized by the space of \emph{ending laminations}.  Let $M$ be a compact $3$-manifold without torus boundary components whose interior admits a complete hyperbolic metric of infinite volume, and let $\rho: \pi_1(M)\to \G$ be a discrete and faithful representation.  Then $M_\rho = \bH^3/\im \rho$ is a hyperbolic $3$-manifold equipped with a  homotopy equivalence $f: M\to M_\rho$ inducing $\rho$ at the level of fundamental groups.  To the conjugacy class of $\rho$ we associate a bounded $3$-cocycle $\hat\omega_\rho$ that defines a class in bounded cohomology called the \emph{bounded fundamental class} of the representation $\rho$.  If  $\sigma: \Delta_3\to M_\rho$ is a singular $3$-simplex, then $\hat\omega_\rho(\sigma)$ is the algebraic volume of the unique totally geodesic hyperbolic tetrahedron homotopic, rel vertices, to $\sigma$ (see \S\ref{boundedfundamtentalclass}).  
In this paper, we attempt to describe the behavior of $[\hat\omega_\rho] \in \Hb_b^3(\pi_1(M);\bR)$ as we let $\rho : \pi_1(M) \to \G$ vary over discrete and faithful representations.  

A result of Yoshida \cite{yoshida:bdd}  shows that if $\phi_1$ and $\phi_2$ are independent pseudo-Anosov homeomorphisms of a closed surface $S$, then the bounded fundamental classes of the infinite cyclic covers of their mapping tori are non-zero and linearly independent.  We prove that if two homotopy equivalent hyperbolic $3$-manifolds have different ending laminations from each other, then their bounded fundamental classes are linearly independent in bounded cohomology. Our analysis only depends on the structure of the ends of hyperbolic $3$-manifolds, and we will always work in the covers corresponding to such ends.  When our manifolds have incompressible ends, we reduce to the case of studying marked Kleinian surface groups, though our techniques apply in the setting of compressible ends, as well. In section \S\ref{main1}, we show
\begin{theorem}\label{main:1} Suppose $\g$ is a finitely generated group that is isomorphic to a Kleinian group without parabolic or elliptic elements, and let $\{\rho_\alpha: \g \to \G: \alpha\in\Lambda\}$ be a collection of discrete and faithful representations without parabolic or elliptic elements such that at least one of the geometrically infinite end invariants of $M_{\rho_\alpha}$ is different from the geometrically infinite end invariants of $M_{\rho_\beta}$ for all $\alpha\not=\beta\in \Lambda$. Then $\{[\hat\omega_{\rho_\alpha}]: \alpha\in\Lambda\}$ is a linearly independent set in $ \Hb_b^3(\g;\bR)$.
\end{theorem}
In fact, our proof shows something slightly stronger.  See Corollary \ref{general} for the more general statement.  In \cite{soma:kleinian}, Soma proves that if $\g$ is finitely generated and $\rho: \g\to \G$ is discrete, does not contain elliptics, and $M_\rho$ has a geometrically infinite end, then $\|[\omega_\rho]\|_\infty = v_3$, where $v_3$ is the volume of the regular ideal tetrahedron.    In \S\ref{main2}, we prove that bounded classes defined by manifolds with geometrically infinite ends are \emph{uniformly separated} in pseudo-norm.
\begin{theorem}\label{separation:1}
Suppose $\g$ is a finitely generated group that is isomorphic to a Kleinian group without parabolic or elliptic elements.   There is an $\epsilon=\epsilon(\g)>0$ such that if $\{\rho_i: \g \to \G: i = 1, ... ,n\}$ is a collection of discrete and faithful representations without parabolic or elliptic elements such that at least one of the geometrically infinite end invariants of $M_{\rho_i}$ is different from the geometrically infinite end invariants of $M_{\rho_j}$ for all $i\not= j$ then 
\[\|\sum_{i = 1}^n a_i[\hat\omega_{\rho_i}]\|_\infty > \epsilon\max{|a_i|}.\]
\end{theorem}

Our proof also produces a criterion to detect faithfulness of representations (see Theorem \ref{injective theorem}).  Note that in the following, we insist neither that $S$ is a closed surface nor that $\rho'$ is discrete. 

\begin{theorem}\label{faithful:1}
Let $S$ be an orientable surface with negative Euler characteristic.  There is a constant $\epsilon = \epsilon(S)$ such that the following holds.  Let $\rho: \pi_1(S)\to \G$ be discrete and faithful, without parabolic elements, and such that $M_\rho$ has at least one geometrically infinite end invariant.  If $\rho': \pi_1(S)\to \G$ is any other representation satisfying \[\|[\hat\omega_\rho] - [\hat\omega_{\rho'}]\|_\infty <\epsilon,\]
then $\rho'$ is faithful.  
\end{theorem}

We state here a rigidity result for marked Kleinian surface groups that follows from Theorem \ref{separation:1}, Theorem \ref{faithful:1}, and an application of the Ending Lamination Theorem.  This result is in the spirit of Soma's \emph{Mostow rigidity for hyperbolic manifolds with infinite volume} (\cite{soma:boundedsurfaces}, Theorem A and Theorem D).   Soma works in the setting of hyperbolic manifolds of infinite volume with bounded geometry.  These manifolds are a countable union of nowhere dense sets in the boundary of the deformation space of marked Kleinian groups of a specified isomorphism type. Soma's lower bound for the separation constant $\epsilon$ depends both on the topology of the surface and the injectivity radii of the two hyperbolic structures being compared, and it tends to zero as the injectivity radii do.  We do not restrict ourselves to working with manifolds with bounded geometry, and our separation constant does not depend on injectivity radii.
\begin{cor}
Let $S$ be a closed orientable surface of genus at least $2$.  There is a constant $\epsilon = \epsilon(S)$ such that the following holds.  Suppose that $\rho:\pi_1(S)\to \G$ is discrete, faithful,  without parabolics, and such that $M_\rho$ has two geometrically infinite ends.    Then for any other discrete representation $\rho': \pi_1(S)\to \G$, if \[\|[\hat\omega_\rho] - [\hat\omega_{\rho'}]\|_\infty < \epsilon,\]
then $\rho$ and $\rho'$ are conjugate.  
\end{cor}

The pseudonorm on degree three bounded cohomology is in general not a norm (see \cite{soma:nonBanach}, \cite{soma:zero}, and more recently \cite{sisto:zero}).  Let $\textnormal{ZN}^3(\g) \subset \Hb^3_b(\g;\bR)$ be the subspace consisting of non-trivial classes with zero pseudonorm.  The reduced space $\overline{\Hb}_b^3(\g;\bR) =\Hb^3_b(\g;\bR)/ \textnormal{ZN}^3(\g)$ is a Banach space.  As a consequence of Theorem \ref{main:1} and Theorem \ref{separation:1}, we can give a more concrete example of how our results can be applied.  
\begin{cor}\label{nice map}There is an injective map \[\Psi: \cEL(S) \to \overline{\Hb}_b^3(\pi_1(S);\bR)\] whose image is a linearly independent, discrete set.
\end{cor}
See \S\ref{main1} for the construction of the map $\Psi$ and for linear independence, and see \S\ref{main2} for discreteness of it's image.  Our techniques also apply to hyperbolic $3$-manifolds with compressible boundary.  Let $H_n$ be a genus $n$ closed handlebody.  The interior of $H_n$ supports many marked hyperbolic structures without parabolic cusps. The geometrically infinite hyperbolic structures without parabolic cusps are parameterized by the set of ending laminations in the Masur domain $\partial H_n$, up to an equivalence by certain mapping classes that reflect the topology of $H_n$.  We call this space of ending laminations $\cEL(\partial H_n, H_n)$ and define it more carefully in \S\ref{infiniteends}.  Indeed, there is an injective map  $\cEL(\partial H_n, H_n) \to \overline{\Hb}_b^3(F_n;\bR)$ whose image is a linearly independent, discrete set.  It is defined analogously to $\Psi$, and we prove linear independence and discreteness in \S\ref{main1} and \S\ref{main2}, respectively.


The plan of the paper is as follows.  In \S\ref{prelims}, we give definitions of bounded cohomology of groups and spaces.  We also review the standard concepts and tools for working with hyperbolic 3-manifolds with infinite volume, setting notation for what follows.  In \S\ref{ELT}, we give an account of some of the tools used in the proof of the positive resolution of Thurston's Ending Lamination Conjecture.  We discuss markings, hierarchies, model manifolds, bi-Lipschitz model maps, and the existence of a class of embedded surfaces with nice geometric properties called extended split level surfaces.  This class of surfaces plays an important technical role in \S\ref{below}, and hierarchies reappear in \S\ref{main2}.  In \S\ref{below}, we prove the main technical lemma that allows us to make uniform estimates for the volume of certain homotopies between surfaces whose diameters may be tending toward infinity.  This is the main difficulty we must overcome while working with manifolds with unbounded geometry.  After establishing the main technical lemma, we obtain  lower bounds for the volumes of classes of straightened $3$-chains that exhaust an end of a manifold.  In \S\ref{sectionabove}, we provide a uniform upper bound for the same $3$-chains when straightened with respect to a different hyperbolic metric on the same manifold.  The lower bounds from \S\ref{below} and the upper bounds from \S\ref{sectionabove} constitute the two main ingredients for the proof of our main result in \S\ref{main1}.  We work first in the context of closed surface groups, and use the Tameness Theorem and the machinery of the proof of the general Ending Lamination Theorem to extend our results to arbitrary, finitely generated Kleinian groups without parabolics.  In \S\ref{main2}, we revisit the hierarchy machinery to give more refined lower bounds for the volume of chains, analogous to those we obtain in \S\ref{below}.  This is where we obtain our second main result controlling the pseudo-norm of linear combinations of bounded fundamental classes.

\section*{Acknowledgments} 
The author would like to thank Kenneth Bromberg for many hours of his time and for his patience.  Helpful conversations with Maria Beatrice Pozzetti inspired the contents of \S\ref{telescopesoftori}, which were useful for improving a result from a previous draft of this manuscript.  We would also like to thank the hospitality of the MSRI and partial support of the NSF under grants DMS-1246989 and DMS-1440140.

\section {Preliminaries}\label{prelims}
\subsection{Bounded Cohomology of Spaces}
Given a connected CW-complex $X$, we define a norm on the singular chain complex of $X$ as follows.  Let $\Sigma_n = \lbrace \sigma: \Delta_n\to X\rbrace$ be the collection of singular $n$-simplices.  Write a simplicial chain $A\in \Cb_n\left( X;\bR\right)$ as an $\bR$-linear combination \[A = \sum \alpha_\sigma\sigma,\] where each ${\sigma\in \Sigma_n}$.  The $1$-norm or \emph{Gromov} norm of $A$ is defined as \[\left\| A \right\|_ 1 = \sum \left| \alpha_\sigma\right|.\]
This  norm promotes the algebraic chain complex $\Cb_\bullet (X ;\bR)$ to a chain complex of normed linear spaces; the boundary operator is a bounded linear operator.  Keeping track of this additional structure, we can take the topological dual chain complex
\[\left(\Cb_\bullet (X;\bR),\partial, \| \cdot  \|_1\right)^*=\left(\Cb^\bullet_b (X;\bR),\delta, \| \cdot \|_\infty\right).\]
   The $\infty$-norm is naturally dual to the $1$-norm, so the dual chain complex consists of \emph{bounded} co-chains.  
Define the \emph{bounded cohomology} $\Hb_b^\bullet (X ;\bR)$ as the (co)-homology of this complex.  For any bounded $n$-co-chain, $\alpha\in \Cb_b^n(X;\bR)$, we have an equality \[\|\alpha\|_\infty = \sup_{\sigma\in \Sigma_n}\left|\alpha ( \sigma) \right|.\] The $\infty$-norm descends to a pseudo-norm on the level of bounded cohomology.  If $A\in \Hb_b^n(X;\bR)$ is a bounded class, the pseudonorm is described by \[\| A \|_\infty = \inf_{\alpha\in A} \|\alpha\|_\infty.\]We direct the reader to \cite{gromov:bdd} for a systematic treatment of bounded cohomology of topological spaces and  fundamental results.  

Matsumoto-Morita \cite{mm:banach} and Ivanov  \cite{ivanov:banach} prove independently that in degree 2, $\|  \cdot \|_\infty$ defines a norm in bounded cohomology, so that the space $\Hb_b^2(X;\bR)$ is a Banach space with respect to this norm.  In \cite{soma:nonBanach}, Soma shows that the pseudo-norm is in general not a norm in degree $\ge 3$.  In \cite{soma:zero}, Soma exhibits a family of linearly independent bounded classes depending on a continuous parameter with arbitrarily small representatives in $\Hb_b^3(S;\bR)$, thus showing that the zero-norm subspace of bounded cohomology in degree 3 is uncountably generated.   This work has recently been generalized in \cite{sisto:zero} by Franceschini \emph{et al.} to prove that the zero norm subspace of bounded cohomology in degree 3 of acylindrically hyperbolic groups is uncountably generated.
\subsection{Bounded Cohomology of Discrete Groups}
Let $G$ be a discrete group.  Then $G$ acts on functions $f: G^{n+1} \to \bR$ by $(g.f)(g_0, ..., g_n) = f(g\inverse g_0, ..., g\inverse g_n)$.  We define a co-chain complex for $G$ by considering the collection of $G$-invariant functions \[\Cb^n(G;\bR) = \lbrace f: G^{n+1} \to \bR: g.f = f, ~\forall g\in G \rbrace.\]  The homogeneous co-boundary operator $\delta$ for the trivial $G$ action on $\bR$ is, for ${f\in \Cb^n(G;\bR)}$, \[\delta f (g_0, ... ,g_{n+1}) = \sum_{i=0}^{n+1} (-1)^i f(g_0, ...,\widehat{g_{i}}, ..., g_{n+1}).\]    The co-boundary operator gives the collection $\Cb^\bullet(G;\bR)$ the structure of a (co)-chain complex.  An $n$-co-chain $f$ is \emph{bounded} if \[\|f\|_\infty = \sup|f(g_0, ..., g_{n})| < \infty,\] where the supremum is taken over all $(n+1)$-tuples $(g_0, ..., g_n) \in G^{n+1}$.  

The operator  ${\delta : \Cb^n_b(G;\bR) \to \Cb^{n+1}_b(G;\bR)}$ is a bounded linear operator with operator norm at most $n+2$, so the collection of bounded co-chains $\Cb^\bullet_b(G;\bR)$ forms a subcomplex of the ordinary co-chain complex. The (co)-homology of $(\Cb_b^\bullet(G;\bR),\delta)$ is called the \emph{bounded cohomology} of $G$, and we denote it $\Hb_b^\bullet (G;\bR)$.  The $\infty$-norm $\|\cdot\|_\infty$ descends to a pseudo-norm on bounded cohomology in the usual way.  Brooks \cite{brooks:remarks}, Gromov \cite{gromov:bdd}, and Ivanov \cite{ivanov:isometric} proved the remarkable fact that for any countable connected CW-complex $M$, the classifying map $K(\pi_1(M),1)\to M$ induces an isometric isomorphism $\Hb^\bullet_b(M;\bR)\to \Hb^\bullet_b(\pi_1 (M);\bR)$.  We therefore identify the two spaces $\Hb^\bullet_b(\pi_1(M);\bR)= \Hb^\bullet_b(M;\bR)$.  

\subsection{The Bounded Fundamental Class}\label{boundedfundamtentalclass}
Fix a Kleinian group $\g\le \G$.  Let $\omega\in \Omega^3(M_\g)$ be such that pullback of $\omega$ by the covering projection $\bH^3 \to \bH^3 / \g =M_\g$ is the  Riemannian volume form on $\bH^3$.  Suppose $\sigma: \Delta_3\to M_\g$ is a singular 3-simplex.  We have a chain map  \cite{thurston:notes}
\[\str : \Cb_\bullet (M_\g) \to \Cb_\bullet (M_\g)\] defined by homotoping $\sigma$,  relative to its vertex set, to a totally geodesic hyperbolic tetrahedron $\str \sigma$.  The co-chain \[ \hat\omega(\sigma)= \int_{\str\sigma}\omega\] measures the signed hyperbolic volume of the straightening of $\sigma$.  
Any geodesic tetrahedron in $\bH^3$ is contained in an ideal geodesic tetrahedron, and there is an upper bound $v_3$ on this volume which is maximized by regular ideal geodesic tetrahedra \cite{thurston:notes}.  Thus \[\displaystyle \| \hat\omega\|_\infty = \sup_{ \sigma: \Delta_3 \to M_\g} |\hat\omega(\sigma)| = v_3.\] 
We use the fact that $\str$ is a chain map, together with Stokes' Theorem to observe that if $\upsilon: \Delta_4 \to M_\g$ is any singular 4-simplex, 
\[\delta \hat\omega(\upsilon) = \int_{\str  \partial  \upsilon  } \omega= \int_{\partial \str \upsilon} \omega = \int_{\str \upsilon} d\omega = 0,\]
because $d\omega \in \Omega^4(M_\g)= \{0\}$.  The class $[\hat\omega]\in \Hb^3_b(M_\g;\bR)$ is the \emph{bounded fundamental class} of $M_\g$.  An isomorphism  $\rho: \pi_1(M)\to \g$ induces a homotopy class of maps $f: M \to M_\g$ such that ${f_* = \rho}$.  Thus we may pullback the class $[\hat\omega]$ by $f$ to obtain ${[\hat\omega_\rho]: = [f^*\hat\omega] \in \Hb^3_b(M;\bR)}\cong \Hb_b^3(\pi_1(M);\bR)$.  

\subsection{The ends of hyperbolic $3$-manifolds} 
A hyperbolic manifold is said to be \emph{topologically tame} if it is homeomorphic to the interior of a compact $3$-manifold.  A hyperbolic manifold with finitely generated fundamental group is said to be \emph{geometrically tame} if all of its ends are geometrically finite or simply degenerate (see definitions below).  Bonahon \cite{bon} proved that any hyperbolic structure supported on the interior of a compact manifold with incompressible boundary is geometrically tame.  Canary \cite{canary:ends} proved that topological tameness implies geometric tameness.  More generally, if $\g\le \G$ is a finitely generated Kleinian group without elliptic elements, by the Tameness Theorem of Agol \cite{agol:tame} and Calegari--Gabai \cite{calegari-gabai:tame} there is a compact three manifold $K_\g$ and a diffeomorphism $h: M_\g=\bH^3/\g \to \interior K_\g$.   Other authors have given accounts of the Tameness Theorem, notably Soma \cite{soma:tameness}, who simplified the key argument of \cite{calegari-gabai:tame}.  Equip $K_\g$ with a Riemannian metric, and let $S$ be a component of $\partial K_\g$.  In this paper, we will ignore parabolic cusps, so $S$ is a closed surface. We apply the Collar Neighborhood Theorem to obtain a neighborhood of $S$ with closure $F_S\subset K_\g$ and a diffeomorphism $F_S\to S\times I $, where $I$ will always denote the closed interval $[0,1]$.  Let $E_S= h\inverse(F_S)$; we have a diffeomorphism $E_S \to S\times [0,\infty)$.  Fix such an identification so that we may think of $S\times [0,\infty)$ as coordinates for $E_S$.  Call $E_S$ a \emph{neighborhood of S}.  We remark that a neighborhood of $S$ is the closure of a neighborhood of an end of $M_\g$ in the usual notion of an end of a topological space.

Define $\cE(\g)$ to be the collection of connected components of the boundary of $K_\g$.  For each $S\in \cE(\g)$, we may find neighborhoods as above $E_S \cong S\times [0,\infty)$ such that $E_S\cap E_{S'}=\emptyset$ if $S\not=S'$.  Then \[\cK(M_\g)=\displaystyle M_\g\setminus \bigcup _{S\in \cE(\g)} \interior E_S\] is diffeomorphic to $K_\g$, and in particular, it's inclusion is a homotopy equivalence.  For each boundary component $S'$ of $\cK(M_\g)$, the inclusion $S'\to M_\g$ is homotopic in $K_\g$ to a unique $S\in \cE(\g)$.  Identify $S$ with $S'$ so that by abuse of notation, if $S\in \cE(\g)$ write $i: S\to M_\g$ to denote the inclusion of $S'\to \partial \cK(M_\g)$ which is homotopic to $ S\to \partial K_\g$ in $K_\g$. 

\subsection{Geometrically finite ends}   
The {\it limit set} $\Lambda_\g$ of a Kleinian group $\g\le \G$ is the set of accumulation points of $\g.x$ for some (any) $x\in \partial \bH^3$.
Denote by $\cH(\Lambda_\g)\subset \bH^3$ the convex hull of $\Lambda_\g$.  The \emph{convex core} of $M_\g$ is $\cC\cC(M_\g) = \cH(\Lambda_\g)/\g$.     

Say that $S\in \cE(\g)$  is {\it geometrically finite} if there is some neighborhood $E_S$  disjoint from $\cC\cC(M_\g)$.  Call $S$ {\it geometrically infinite} otherwise.
Since $K_\g$ is compact, the collection $\cE(\g)$ is finite; enumerate its elements $\cE(\g) = \{ S_1, ... ,S_{k(\g)}\}$. 
By the Ending Lamination Theorem (see \S\ref{ELT}) the isometry type of $M_\g$ is uniquely determined by its topological type together with its {\it end invariants}  $(\nu(S_1), ... , \nu(S_{k(\g)}))$.  We give a description of the end invariant $\nu(S)$.

Suppose $S\in \cE(\g)$ is geometrically finite.  A Kleinian group $\g$ acts properly discontinuously on its {\it domain of discontinuity } $\Omega_{\g} = \hat\bC \setminus \Lambda_\g$, and $\Omega_{\g}\not=\emptyset$ when $S$ is geometrically finite. Indeed a component $\tilde{S}$ of $\Omega_\g$ carries an action by $\g$ which is free and properly discontinuous by conformal automorphisms, inducing a conformal structure $X=\tilde{S}/ \g$ on $S$.  Moreover, $X$ is a boundary component of the manifold $\overline{M_\g} = (\bH^3\cup\Omega_\g) /\g$.
The inclusion $i:S\to M_\g$ is homotopic in $\overline{M_\g}$ to inclusion $S\to X$; thus $(i, X)$ defines a point in Teichm\"uller space $\teich(S)$. To each geometrically finite component $S$, we associate the end invariant $\nu(S) = (i,X)$. 
 
\subsection{Geometrically infinite ends}\label{infiniteends}
Suppose $S\in \cE(\g)$ is geometrically infinite. Then $E_S$ is \emph{simply degenerate}.  That is, there is a sequence $\{\gamma_i'\}$ of homotopically essential, closed geodesic curves exiting $E_S$.  Each $\gamma_i'$ is homotopic in $E_S$ to a simple closed curve $\gamma_i\subset S\times\{0\}$.   Moreover, we may find such a sequence such that the length $\ell_{M_\g}(\gamma_i')\le L_0$, where $L_0$ is the Bers constant for $S$.  Equip $S=S\times \{0\}$ with any hyperbolic metric.  Find geodesic representatives $\gamma_i^*\subset S$ with respect to this metric.  The sequence $\gamma_i^*$ converges in the Hausdorff topology on closed sets of $S$ to a geodesic lamination $\lambda_S'\subset S$.  By the classification theorem for geodesic laminations, $\lambda_S'$ contains at most finitely many isolated leaves $\ell_1, ..., \ell_k$.  Let $\lambda_S = \lambda_S'\setminus \{\ell_1, ..., \ell_k\}$.   Thurston \cite{thurston:notes}, Bonahon \cite{bon}, and Canary \cite{canary:ends} show, in various contexts, that $\lambda_S$ is minimal, filling, and does not depend on the sequence $\gamma_i$.  Moreover, for any two hyperbolic structures on $S$, the spaces of geodesic laminations are canonically homeomorphic.  This \emph{ending lamination} is the end invariant $\nu(S) = \lambda_S$.  Call $\cEL(S)$ the set of minimal, filling geodesic laminations.  

In the case that the end $S$ is \emph{compressible}, \ie the inclusion $S\to K_\g$ does not induce an injection on $\pi_1$, there is more to say.   A meridian $m\subset S\subset \partial K_\g$ is a homotopically non-trivial simple closed curve in $S$ which bounds a disk in $K_\g$ (and hence is nullhomotopic, there).  A lamination $\lambda\subset S$ is said to belong to the \emph{Masur domain} if $\lambda$ is not contained in the Hausdorff limit of any sequence of meridians.  Let $\Mod_0(S,M_\g)\le \Mod(S)$ be the group of homotopy classes of orientation preserving self homeomorphisms of $S$ which extend to homeomorphisms of $K_\g$ homotopic to the identity on $K_\g$.  If two laminations $\lambda, \lambda'\in \cEL(S)$  differ by a mapping $\varphi\in \Mod_0(S,M_\g)$, then they  are indistinguishable in $M_\g$. That is, if  $\lambda'=\varphi(\lambda)$ and if $\alpha_j$ is a sequence of curves converging to $\lambda$ in $S$, then ${\varphi(\alpha_j)\to \varphi(\lambda)=\lambda'}$, as $j\to \infty$.  However, for each $j$, the geodesic representatives of these curves coincide in $M_\g$, \ie ${i(\alpha_j)^*=i(\varphi(\alpha_j))^*\subset M_\g}$.    

Below, a theorem of Namazi--Souto supplies a converse to the observation above, but first we need a definition.  Let $S_0$ be a closed, oriented surface of negative Euler characteristic, and suppose $h: S_0\to M_\g$ is continuous.  
 A \emph{pleated surface} is a length preserving map $f: X\to M$ from a hyperbolic surface $X\in \teich(S_0)$ to $M_\g$ homotopic to $h$ and such that every point $p\in X$ is contained in a geodesic segment which is mapped isometrically.  A lamination $\lambda\subset S_0$ is \emph{realized} in $M_\g$ if there is a pleated surface $f: X\to M_\g$ that maps every leaf of $\lambda$ to a geodesic in $M_\g$.

\begin{theorem}[\cite{namazi-souto:density}, Theorem 1.4] \label{NS}
Suppose that $\phi_0:M_{\g_1}\to M_{\g_2}$ is a homotopy equivalence of hyperbolic 3-manifolds and that that $\lambda$ is a filling, minimal Masur domain lamination on an end $S\in \cE(\g_1)$ such that  $\phi_0(\lambda)\subset M_{\g_2}$ is \emph{not} realized.  Then $\phi_0$ is homotopic, relative to the complement of a regular neighborhood of $i(S)$ to a map ${\phi: M_{\g_1}\to M_{\g_2}}$ such that $\phi$ restricts to a homeomorphism $S\to S'$ for some $S' \in \cE(\g_2)$ and $\nu(S')=\phi(\lambda)$.
\end{theorem}
With notation as in Theorem \ref{NS},  work of Otal \cite{otal:thesis} implies that if $\lambda$ is a minimal, filling Masur domain lamination on $S\in \cE(\g_1)$ that is not realized by a pleated surface homotopic to $\phi_0|_{S}:S\to M_{\g_2}$, then if $\alpha_j\subset S$ is any sequence of simple closed curves converging to $\lambda$, then $\{\phi_0(\alpha_j)^*\}$ must leave every compact set in $M_{\g_2}$.  For an account of Otal's work and some generalizations, see \cite{KS:nonrealize} \S4.

We will think of the end invariant for a geometrically infinite compressible end $S\in \cE(\g)$ as an equivalence class of minimal, filling Masur domain laminations up to the action of $\Mod_0(S,M_\g)$; call this collection $\cEL(S,M_\g)$.  We associate to $S$ the end invariant $\nu(S) = \lambda_S \in \cEL(S,M_\g)$.  

\subsection{Marked Manifolds}
Suppose that $N_0$ and $N_1$ are complete hyperbolic 3-manifolds without parabolic cusps, and $f_i: M_\g \to N_i$ is a homotopy equivalence for each $i=0, 1$.  Say that $(f_0,N_0)\sim (f_1,N_1)$ are equivalent if there is an orientation preserving isometry $\phi : N_0 \to N_1$ such that $\phi\circ f_0$ is homotopic to $f_1$.  Define $\Hyp_0(M_\g)$ to be a collection of all such pairs up to equivalence.  

For each $i$, $\pi_1(N_i)$ acts on $\bH^3$ as the group of deck transformations for the cover $\bH^3\to N_i$, which is defined up to conjugacy in $\G$.  Thus $\pi_1(N_i) = \g_i \le \G$ is Kleinian.  The marking $f_i$ induces an isomorphism $\rho_i = {f_i}_* : \pi_1(M_\g)=\g \to \g_i$ where $\g_i$ contains no parabolics, because $N_i$ had no parabolic cusps.  So we may also think of $\Hyp_0(M_\g)$ as the collection of discrete, faithful representations without parabolics $\rho: \g\to \G$, up to conjugacy in $\G$.

\subsection{Simplicial Hyperbolic Surfaces}

A \emph {triangulation} $\cT$ of a closed surface $S$, is a 3-tuple $\cT = (V, A, T)$.  $V$ is a finite set of vertices.  The collection $A= \{\alpha_i : \Delta_1\to S\}\subset \Cb_1(S)$ is a maximal arc system for $(S,V)$;  $T\subset \Cb_2(S)$ is a collection of 2-simplices satisfying the relevant properties.  A triangulation determines a 2-chain representing the fundamental class $[S]\in \Hb_2(S;\bZ)$. By abuse of notation, we write  $\cT= \displaystyle \sum_{\tau \in T}\tau\in [S]$.

A \emph{pants decomposition} $\cP$ of $S$ is a maximal collection of pairwise non-homotopic, disjoint, simple closed curves on $S$.  We will be interested in triangulations $\cT_\cP$ with $3g-3$ vertices so that for each $\alpha\in \cP$, there is exactly one vertex $v_\alpha$ and one arc joining $v_\alpha$ to itself in the homotopy class of $\alpha$.  Call such a $\cT_\cP$ \emph{adapted} to $\cP$.   The complement of $\cP$ is a disjoint union of $2g-2$ subsurfaces homeomorphic to a three-times punctures sphere, which we call a \emph{pair of pants}.  
By a simple Euler characteristic computation, for any pants decomposition $\cP$, there is a triangulation $\cT_\cP= (V, A, T)$ adapted to $\cP$  such that $|V|=3g-3$, $|A| = 15g-15$ and $|T| = 10g-10$.   

\begin{figure}[h]
\includegraphics[height = 1in]{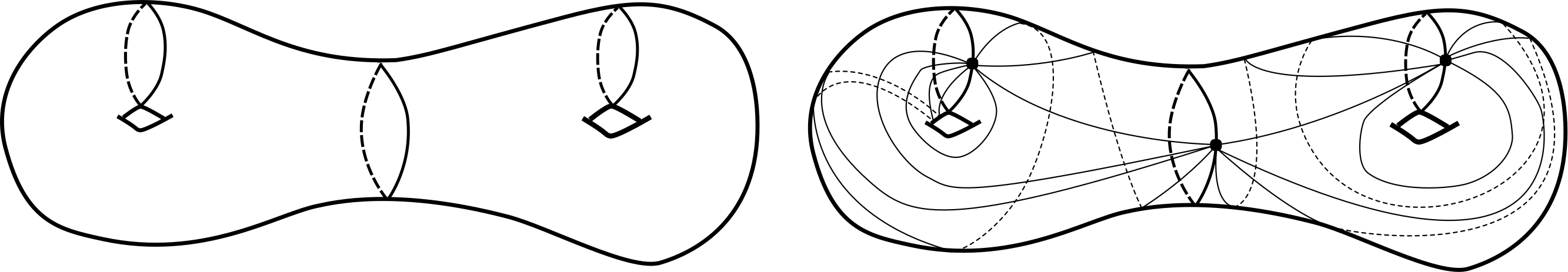}
\caption{A pants decomposition $\cP$ of $S$ and a triangulation $\cT_\cP$ adapted to $\cP$}
\label{triangulationsfigure}
\end{figure}

The following definitions are due to Bonahon \cite{bon}.  A \emph {simplicial pre-hyperbolic surface} is a pair $(F,\cT)$ where $F :S\to M_\g$ is a continuous map; $\cT$ is a triangulation so that for each $\alpha\in A$, ${\im F\circ \alpha \subset M_\g}$ is a geodesic segment, and for each $\tau\in \cT$, $\im  F\circ \tau$ is totally geodesic.  We call $\cT$ the \emph{triangulation associated} to $F$.  A simplicial pre-hyperbolic surface is a \emph{simplicial hyperbolic surface} if the cone angle about each vertex in the metric on $S$ induced by $F$ is at least $2\pi$.  
If $\cT_\cP$ is adapted to $\cP$, and $(F',\cT_\cP)$ is a pre-simplicial hyperbolic surface, we may homotope $F'$ to a map $F$ such that for each $\alpha\in \cP$, $F$ maps $\alpha$ to it's geodesic representative, \ie such that $F(\alpha) = F(\alpha)^*$, and $(F,\cT_\cP)$ is simplicial hyperbolic.  Notice here that for any (pre-)simplicial hyperbolic surface $(F,\cT)$, \[ \str F_*\cT = F_*\cT.\]

From now on, given a continuous map $F$, we will let $F_*=F$ denote the map induced on chains.  We break with this convention for emphasis.  We will make use of the following theorem for simplicial hyperbolic surfaces,   due to Bonahon \cite{bon}.
\begin{thm}[Bounded Diameter Lemma]
For any compact set $K\subset M_\g$, there is a compact set $K'\subset M_\g$ such that if $F: S\to M_\g$ is a simplicial hyperbolic surface and $\im F \cap K\not=\emptyset$, then $\im F\subset K'$.
\end{thm}

\section{Hierarchies and Models}\label{ELT}

If $\rho: \pi_1(S)\to \g\le \G$ is discrete and faithful, then $\rho$ induces a homotopy class $[f]$ of orientation preserving mappings $f: S\times \bR \to M_\rho$.  The ends $\cE(\g)$ are the surfaces $S^-=S\times \{-\infty\}$ and $S^+ = S\times \{+\infty\}$ which are the boundary components of the compact manifold $S\times [-\infty, +\infty]$.  Define ${\nu= (\nu(S^-),\nu(S^+))=(\nu^-,\nu^+)}$.

The machinery used in the proof of Thurston's Ending Lamination Theorem is extremely useful for understanding the structure of geometrically infinite ends of hyperbolic $3$-manifolds.  This machinery arises as the result of a detailed analysis of the geometry of the \emph{curve complex} $\cC(S)$ of $S$ whose vertices are non-trivial homotopy classes of simple closed curves.  A $k$-simplex spans vertices $v_0, v_1, ..., v_k$ if the curves they represent can be realized disjointly on $S$.  The curve complexes of subsurfaces and subsurface projections  also play an important role.  
There is a hierarchical structure to the curve complex; Masur and Minsky develop a theory of hierarchies of tight geodesics in \cite{masur-minsky:complex1} and \cite{masur-minsky:complex2}.  From the end invariants $\nu$, Minksy in  \cite{minsky:ELTI}  constructs a hierarchy $H_\nu$, and from the hierarchy, a model manifold $M_\nu$ made out of blocks and tubes.   The thick part of the model manifold admits a Lipschitz mapping into $\cC\cC(M_\rho)$.  Brock--Canary--Minsky \cite{brock-canary-minsky:ELTII} extend the model map to the thin parts of $M_\nu$ and promote it into a $K$-bi-Lipschitz homeomorphism in the homotopy class determined by $\rho$, where $K=K(S)$.  We state here a special case of the Ending Lamination Theorem.

\begin{thm}[Ending Lamination Theorem for Closed Surface Groups Without Parabolics]
Let $\g\le \G$ be a Kleinian surface group, and $(f_0,N_0), (f_1,N_1)\in \Hyp_0(M_\g)$.  There is an orientation preserving isometry $\phi :N_0\to N_1$ such that $\phi \circ f_0\sim f_1$ if and only if $\nu(N_0) = \nu(N_1)$.
\end{thm}

In \S\ref{markings}-\ref{subhierarchies}, we describe curve complexes of subsurfaces, markings, and  finite hierarchies of tight geodesics  due to Masur-Minsky \cite{masur-minsky:complex2}.  We will also explain how to restrict to subpaths in a hierarchy which will retain their hierarchical structure as in \cite{durham:hierarchies} Appendix 8.  Since we use some of these properties in \S7, we are more careful with our language than in the rest of this section.
We discuss the construction of the model manifold $M_\nu$ from a hierarchy $H_\nu$  loosely, making precise statements when we will require them later.  There are two main features of these constructions that will interest us most.  We will make use of the existence of a class of embedded surfaces, called \emph{extended split level surfaces} whose geometry is essentially bounded.  They will play a crucial role in \S4.  We will also need the description of the model as a collection of blocks and tubes.  We return to hierarchical constructions and collect some counting results about hierarchies in \S7.  

\subsection{Curve complexes and markings}\label{markings}
For $Y\subset S$ a homotopically essential subsurface define $\xi(Y) = 3g+b$, where $g$ is the genus of $Y$ and $b$ is the number of its boundary components or punctures. In the case that $\xi(Y)\ge 3$, the vertices of the curve complex $\cC(Y)$ are non-trivial homotopy classes of simple closed curves that are not homotopic to a boundary component of $Y$.  Vertices $v_0, ..., v_k$ span a $k$-simplex if the curves they represent can be realized with pairwise minimal intersection number on $Y$.  If $\xi(Y) =4$, then $Y$ is a once punctured torus or a $4$ times punctured sphere, and the minimal intersection numbers are $1$ and $2$ respectively.  In both cases, $\cC(Y)$ is identified with the Farey graph.  The curve complex of a sphere with $3$ boundary components is empty, and the curve complexes associated to annular subsurfaces require more care to define. In order to do so, we fix a hyperbolic metric on $S$ for reference.  If $\gamma\subset S$ is an essential simple closed curve, we consider a collar neighborhood $Y_\gamma$ of $\gamma$ that embeds isometrically into the cover $\widetilde Y_\gamma \to S$ corresponding to $\langle \gamma \rangle \le \pi_1(S)$.  There is a natural compactification $\widehat Y_\gamma$ of $\widetilde Y_\gamma$ obtained as a quotient of the closed disk away from the fixed points of $\langle \gamma\rangle $ acting by deck transformations.  The vertices of the annular curve complex $\cC(Y_\gamma)$ are the simple geodesic arcs joining one boundary component of $\widehat Y_\gamma$ with the other.   Edges correspond to disjointness.  Suppose $\alpha, \gamma \in \cC(S)$ are vertices with geodesic representatives $\alpha^*$ and $\gamma^*$ and intersection $i(\alpha^*, \gamma^*)>0$.  We define a projection $\pi_\gamma(\alpha)\subset \cC(Y_\gamma)$ by taking the set of lifts of $\alpha^*$ that cross the core curve of $\widehat Y_\gamma$.  This collection has diameter $1$ where  for $Y\subset S$, we endow the $1$-skeleton of curve complexes $\cC(Y)$ with the path metric $d_{\cC(Y)}$, giving all edges length one.  

A \emph{marking} $\mu$ on  $S$ consists of a collection $\base(\mu) = \{\gamma_1, ..., \gamma_n\}$ of \emph{base curves} which forms a diameter (at most) one subset of $\cC(S)$, and to each base curve $\gamma_i$, a diameter (at most) one set of vertices $t_i\subset \cC(Y_{\gamma_i})\cup\{\emptyset\}$ in the curve complex of the annular surface $Y_{\gamma_{i}}$.  Such a $t_i$ is called a \emph{transversal}. We write $(\gamma_i,t_i)$ to denote a base curve with its associated transversal.  Say that $\mu$ is \emph{complete} if $\base(\mu)$ is maximal, \ie $\base(\mu)$ is a pants decomposition of $S$, and every curve has a (non-empty) transverse curve.  Call a marking \emph{clean} if each pair $(\gamma,t)$ is of the form $(\gamma, \pi_\gamma(\alpha))$ where $\alpha \in \cC(S)$, a regular neighborhood of $\gamma\cup \alpha$ fills a surface $F$ with $\xi(F) =4$ in which $d_{\cC(F)}(\gamma,\alpha)=1$, and $\alpha$ does not intersect any curve in $\base(\mu)$ other than $\gamma$.  If $\mu$ is a complete marking, there is a (uniformly finite) collection of clean markings which are \emph{compatible} with $\mu$.  We say $\mu$ is \emph{compatible} with a clean marking $\mu'$  if they share the same base curves, and for each pair $(\alpha,t)\in \mu$ and $(\alpha,t')\in \mu'$, the distance between $t$ and $t'$ is minimal in the curve complex of the annular subsurface $Y_\alpha$.   Thus, given a complete marking $\mu$, as long as we are willing to live with a bounded amount of ambiguity (which we usually are), we may choose a compatible complete clean marking $\mu'$ almost canonically.  We will also consider \emph{generalized markings}, which in this paper are actually just ending laminations $\lambda\in \cEL(S)$.  

\subsection{Hierarchies of tight geodesics}
Suppose $Y\subset S$ has $\xi(Y)>4$.  A (possibly bi-infinite) sequence of diameter 1 subsets $..., v_1, v_2, ... \subset \cC(Y)$ is called \emph{tight} if \begin{enumerate}
\item For each $0\le i <j \le n$, and vertices $v_i'\in v_i$ and $v_j'\in v_j$, $d_Y(v_i',v_j') = j-i$.
\item For each $0<i<n$, $v_i=\partial F(v_{i-1},v_{i+1})$ where $F(v_{i-1},v_{i+1})\subset Y$ is the surface filled by $v_{i-1}\cup v_{i+1}$.  
\end{enumerate}
A sequence in a subsurface $Y$ where $\xi(Y)=2$ or $4$ is tight if it is the vertex sequence of a geodesic in $\cC(Y)$.  
A finite \emph{tight geodesic} $g$ in $\cC(Y)$ is the data of a tight sequence $v_0, ..., v_n$ and a pair of markings $\mathbf{I}(g)$ and $\mathbf{T}(g)$ called the initial and terminal markings respectively with the condition that $v_0\subset \base(\mathbf{I}(g))$ and $v_n\subset \base(\mathbf{T}(g))$.  The \emph{domain} or \emph{support} $D(g)$ of $g$ is the surface $Y$.  
If $g$ is a tight geodesic supported on $Y$ then $X\subset Y$ is a \emph{component domain} for $g$ if there is a simplex $v$ of $g$ such that $X$ is a component of $Y\setminus v$ or $X$ is an annular subsurface with core curve a vertex of $v$.  

Roughly, a \emph{hierarchy of tight geodesics $H$ on $S$} is a collection of tight geodesics supported on subsurfaces of $S$ satisfying a collection of compatibility and completeness conditions.  First there must be a \emph{main geodesic} $g_H\in H$ such that $D(g_H)=S$; the initial and terminal markings for $H$, denoted $\mathbf{I}(H)$ and $\mathbf{T}(H)$, are $\mathbf{I}(g_H)$ and $\mathbf{T}(g_H)$ respectively.  If $Y$ is a component domain for geodesics $g,h\in H$ and $Y$ contains a vertex in a simplex of $h$ and of $g$,  then there is a geodesic $k\in H$ such that $D(k)=Y$.  Moreover, the initial and terminal markings for $k$ must contain curves in a simplex of $g$ and $h$.  Finally, for each $k\in H\setminus\{g_H\}$, $k$ has support which is a component domain for some $h, g \in H$ and all of their markings must be compatible (see \cite{masur-minsky:complex2}, \S4).  

If $H$ is a hierarchy, and $\mathbf{I}(H)$ and $\mathbf{T}(H)$ are both complete markings, then we say that $H$ is \emph{complete}.    Complete hierarchies exist between any two complete markings.

Let $H$ be a complete hierarchy.  A \emph{slice} in $H$ is a set $\tau$ of pairs $(h,v)$, where $h\in H$ and $v$ is a simplex of $h$ satisfying the following conditions:
\begin{enumerate} 
\item A geodesic $h$ appears in at most one pair in $\tau$.
\item There is a vertex $v_\tau$ of $g_H$ such that $ (g_H, v_\tau)\in \tau$.
\item For every $(k,w)\in \tau\setminus\{(g_H,v_\tau)\}$, $D(k)$ is a component domain of $(D(h),v)$ for some $(h,v)\in \tau$.
\end{enumerate}
If a slice $\tau$ satisfies the following condition, we say that $\tau$ is \emph{complete}: Given $(h,v)\in \tau$, for every component domain $Y$ of $(D(h),v)$ with $\xi(Y)\not=3$ there is a pair $(k,w)\in \tau$ with $D(k)=Y$.  Let $V(H)$ be the collection of complete slices of $H$.  

For each $\tau\in V(H)$, the curves in the collection $\{v : (h,v)\in \tau \text{ and } \xi(D(h))\not=2\}$ are disjoint, and form a pants decomposition $\cP_\tau$ of $S$.  By completeness, for each $\alpha\in \cP_\tau$, there is a pair $(h_\alpha, v_\alpha)\in \tau$ such that the support of $h_\alpha$ is the annular subsurface with core curve $\alpha$.  Thus, each complete slice uniquely defines a complete marking $\mu_\tau$ with $\base(\mu_\tau)=\cP_\tau$.   The set $V(H)$ admits a strict partial order $\prec$ with minimal and maximal elements (the slice corresponding to) $\mathbf{I}(H)$ and $\mathbf{T}(H)$ respectively.  A \emph{resolution} $\{\tau_i\}^N_{i=0}$ is a maximal chain with respect to this order.  That is, \[\mathbf{I}(H) = \tau_0 \prec \tau_1 \prec ... \prec \tau_N=\mathbf{T}(H).\]
Informally, a resolution is a way to sweep through all of the curves `in order.'  The relation $\tau_i\prec\tau_{i+1}$ represents definite, forward progress along $H$, and every curve in the hierarchy is in some slice in any resolution.  Let $J = [0,N]\cap \bZ$.  We will associate to the resolution the sequence of markings $\{\mu_{\tau_i}\}_{i \in J }$ and we will call the pants decompositions $\cP_i = \base(\mu_{\tau_i})$.  The sequence $\{\cP_i\}_{i\in J}$ forms a path in the pants graph $\mathbf{P}(S)$ with repetitions.   For details, see  \S5 of \cite{masur-minsky:complex2}.

\subsection{Subhierarchies}\label{subhierarchies}
We now explain how to truncate a hierarchy $H$ to recover a `subhierarchy.'  We direct the reader to \cite{durham:hierarchies} Appendix 8 for more details of this construction.  Fix a resolution of $H$ by slices and recover markings $\{\mu_i'\}_{i\in J}$, and find complete, clean markings $\mu_i$ compatible with $\mu_i'$ for all $i$.  A hierarchy of tight geodesics coarsely defines a \emph{hierarchy path} $P_H$ in the marking complex $\cM(S)$  of $S$ whose vertices are complete, clean markings on $S$; there is an edge between two markings if they are related by an \emph{elementary move}.  The elementary moves are \begin{enumerate}
\item A \emph{twist} move: replace some pair $(\gamma,t)\in \mu$ with $(\gamma, T^\pm_\gamma(t))$, where $T^\pm_\gamma$ is either a right or left hand Dehn twist about $\gamma$. 
\item A \emph{flip} move:  replace the pair $(\gamma, t)$ with $(t,\gamma)$; the marking may no longer be clean, so one must then find a clean, compatible marking.
\end{enumerate}
 We direct the reader to (\cite{masur-minsky:complex2}, \S5) for more about the marking complex and the paths $P_H: J \to \cM(S)$ connecting the complete, clean markings compatible with $\mathbf{I}(H)$ and $\mathbf{T}(H)$.  
 
 Assume $i<j$ are in the interval $J$, and call $\mu_i = P_H(i)$ and $\mu_j = P_H(j)$.   We will now outline a procedure to build a hierarchy $H_{i,j}$ such that $\mathbf{I}(H_{i,j})= \mu_i$,  $\mathbf{T}(H_{i,j}) = \mu_j$ and $H_{i,j}$ carries a resolution that is just restriction of the original resolution of $H$.  
  
Suppose $Y$ is the support of a geodesic $h\in H$, with first simplex $v_0$ and last simplex $v_n$.  The \emph{active segment} for $Y$ is the interval $J_Y =[i_Y, t_Y]\cap \bZ\subset J$ where
\[\text{$i_Y= \min\{i' : P_H(i')$ contains $v_0\}$}\]
 and \[\text{$t_Y= \max\{i': P_H(i')$ contains $v_n\}$}.\]
 Let $D$ be the collection of subsurfaces of $S$ whose active intervals have non-empty intersection with $[i,j]$.  That is,
\[D = \{Y\subset S: \text{there is an $h\in H$ such that $D(h)=Y$ and $ J_Y\cap [i,j]\not=\emptyset$}\};\] 
$D$ is the collection of subsurfaces which will participate in the new hierarchy $H_{i,j}$. Suppose $D(h)\in D$, and write the tight sequence $h = v_0, ..., v_n$.  Let $f(h) = \min\{k : v_k \subset \mu_i \}$ if that set is non-empty, and $0$ otherwise.  Similarly, let $\ell(h) = \max\{k: v_k\subset \mu_j\}$ if that set is non-empty and $n$ otherwise.  Define the \emph{truncated tight sequence for $h'$ for $Y$} is the tight sequence  $v_{f(h)}, v_{f(h)+1}, ..., v_{\ell(h)}$.  Durham defines initial and terminal markings on each truncated tight sequence $h'$ by induction stipulating that the main geodesic has initial and terminal markings $\mu_i$ and $\mu_j$.  Lemma 8.3.1 of \cite{durham:hierarchies} states that the collection $H_{i,j} = \{h' : h\in D\}$ of tight geodesics  is a hierarchy.  Lemma 8.3.2 states that $\{\mu_k': k\in [i,j]\cap \bZ\}$ is a resolution of $H_{i,j}$; \ie $\{\mu_k'\}$ restricts to a resolution on $H_{i,j}$.  We will use these facts in \S7.

\subsection{From end invariants to hierarchies}\label{invnts}

Given a marked Kleinian surface group $\rho: \pi_1(S) \to\G$ without parabolics, we first recover its end invariants $\nu(\rho) = (\nu^-, \nu^+)$.  If either end invariant is geometrically finite, then we replace the Teichm\"uller parameter $X$ with a marking $\mu_X$ such that $\cP_X = \base(\mu_X)$ is a pants decomposition in which every curve has length at most $L_0$---the Bers constant for $S$.  The transversal $t$ for $\gamma\in \cP_X$ is chosen to be of the form $t= \pi_\gamma(\alpha)$, where $\ell_X(\alpha^*)$ is minimal among curves that have essential intersection with $\gamma$.  If an end invariant is infinite, it is an ending lamination, and so it is a generalized marking, and we will be most interested in the case where at least one of the end invariants of $M_\rho$ is an ending lamination of $S$.  Masur-Minsky \cite{masur-minsky:complex1} proved that $\cC(S)$ is Gromov hyperbolic.  Its Gromov boundary is therefore well defined, and Klarreich \cite{klarreich:boundary} proved that there is a natural identification of $\partial \cC(S)$ and $\cEL(S)$.  So a pair of end invariants gives a coarsely defined pair in $\cC(S)\cup \partial\cC(S)$.  In  \cite{masur-minsky:complex2}, Masur-Minsky proved that finite hierarchies have certain local finiteness properties.  Minsky uses these finiteness properties to extract infinite limits of finite hierarchies whose initial and terminal markings are generalized markings. That is, there is a hierarchy of tight geodesics $H_\nu$ joining the data $(\nu^-,\nu^+)$.     

To the hierarchy $H_\nu$, there corresponds a \emph{model manifold} $M_\nu$ built from a union of \emph{blocks} and \emph{tubes}.  Each block is either an internal block or a boundary block.  There are exactly two isometry types of internal blocks, while the boundary blocks are more complicated; we will not need them.

\subsection{The model metric}
 We describe the metric endowed to each internal block and the instructions for gluing blocks together. Let $\epsilon(i)$ be the Margulis constant for $i$-dimensional hyperbolic space.  Fix a number $\epsilon_0<\epsilon(3)$, and find a number $\epsilon_1<1$, depending on $\epsilon_0$ such that if ${F : S\to M_\rho}$ is an incompressible pleated surface, the $\epsilon_0$ thick part of $S$ maps into the $\epsilon_1$ thick part of $M_\rho$ (see \cite{minsky:ELTI}, \S3.2.2).   Let  $A_{\gamma} =  \bH^2/ \langle z \to e^{\frac{\epsilon_1}{2}}z\rangle$ and $\gamma$ be the geodesic representative of the core curve of $A_\gamma$.  Let $\collar(\gamma)$ be the collar neighborhood of $\gamma$ whose boundary components have length $\epsilon_1$.  Call an annulus \emph{standard} if it is isometric to $\cl(\collar(\gamma))$.  Endow a pair of pants $Y$ with the unique hyperbolic metric that makes each of its boundary components length $\epsilon_1/2$.  Double $Y$ along its boundary, and remove the interior of the three standard annuli.  If a pair of pants $Y'$ is isometric to either of these components, call it \emph{standard}.  Let $W_1$ be a four holed sphere and $W_2$ be a one holed torus. Let $\gamma_i^+, \gamma_i^-$ be adjacent in the curve graph of $W_i$ so that their intersection number is minimal among homotopically distinct, nontrivial curves, and let $T_i^\pm\subset W_i$ be collar neighborhoods of $\gamma_i^\pm$.  We build blocks \[B_i = (W_i \times [-1,1])\setminus (T_i^-\times [-1,-1/2) \cup T_i^+ \times (1/2,1]).\]  The boundaries of these blocks decompose as a union of annuli and pairs of pants. Call $(W_i\times \{\pm1\}\setminus T^\pm)\subset \partial B_i$ the \emph{gluing boundary}.  It consists of a total of $4$ pairs of pants.   Endow $B_i$ with fixed metrics $\sigma_i$ such that 
 \begin{enumerate}[(i)]
 \item $\sigma_i$ restricts to a standard metric on every pair of pants $Y$ in the gluing boundary of $B_i$
  \item $\sigma_i$ restricts to a euclidian annulus $S^1\times [0,\epsilon_1]$ on every component $\partial W_i\times [-1,1]$ and the length of a level circle is $\epsilon_1$.  
  \item $\sigma_i$ restricts to a euclidian annulus $S^1\times [ 0, \epsilon_1/2]$ on the boundary components $\partial(T_i^+\times [1/2,1])\cap B_i$ and $\partial (T_i^-\times [-1,-1/2])\cap B_i$ where the length of a level circle is $\epsilon_1$. 
  \item The product structures on annuli in items (ii) and (iii) agree with the product structure induced by the inclusion in $W_i\times [-1,1]$. 
  \end{enumerate}
See \cite{minsky:ELTI},  \S8.1 for more on the construction of blocks and \S8.3 for more on the metrics $\sigma_i$.
\subsection{Building the model}
Fix a resolution of $H$ by slices and recover markings $\{\mu_j\}_{j \in J}$.   If $k\in J$ is such that ${d_{\mathbf{P}(S)}(\base(\mu_k),\base(\mu_{k+1}))=1}$, then there corresponds a block $B(k)$ whose gluing boundaries are components of $S\setminus \base(\mu_k)$ and $S\setminus \base(\mu_{k+1})$.  In fact, for each edge $e$ in a geodesic $h\in H$ such that $\xi(D(h))=4$, there there is a corresponding block $B(e)$ in the model. Distinct blocks $B$ and $B'$ are identified along pairs of pants $Y$ in the gluing boundary if they represent the same subsurface in $S$. Suppose $Y$ is a pair of pants which is a component of  $S\setminus \base(\mu)$ where $\mu$ is a slice of our resolution.  Then $Y$ is in the gluing boundary of exactly 2 blocks, and an inductive argument using the resolution of $H_\nu$ yields an embedding of the union of blocks $M_\nu[0]= \cup_\partial B(e)$ into $S\times \bR$.  A pair of pants $Y$ in the gluing boundary of a block is mapped to a level surface $Y\times \{t_Y\}$.  Given the base curves $\cP$ for a slice of our resolution, a \emph{split level surface} $F_\cP$ is a union over pairs of pants \[F_\cP = \bigcup_{Y\subset S\setminus \cP} Y\times \{t_Y\}\]
equipped with its model metric.  That is, each component is a standard pair of pants.  

The complement of the thick part of the model $M_\nu[0]$ in $S\times \bR$ is a collection of solid tori called \emph{tubes}.  Each tube is of the form $U(v)=\collar(v)\times I$, where $v$ is a vertex of the hierarchy $H_\nu$; in the path metric on $M_\nu[0]$, $\partial U(v)$ is isometric to a euclidean torus.  The euclidean structure on  $\partial U(v)$ extends uniquely to a hyperbolic metric on $U(v)$, and we endow $U(v)$ with this metric.  
Call the collection of metrized tubes $\cU$, and let $M_\nu = M_\nu[0]\sqcup \cU$.  Let $F_\cP$ be a split level surface.  If $F_\cP$ meets $\partial U(v)$ for some $v$, then the intersection is a pair of longitudinal geodesics in $\partial U(v)$ of length $\epsilon_1$ by construction.  Each curve in the pair is isotopic to the geodesic closed curve $v^*\subset U(v)$; parameterize $v^*$ proportionally to arclength as well as each of these longitudes.  We construct a geodesic ruled annulus  for each of the two longitudes and glue these annuli along $v^*$.  Let $S_\cP$ be the union, identified along their common boundary, of $F_\cP$ with the geodesic ruled annuli joining the boundary of $Y$ with $v^*$ for each $v\in \cP$.  Call $S_\cP$ an \emph{extended split level surface};  $S_\cP$ is identified with $S$ by inclusion and is isotopic in $M_\nu$ to $S\times \{0\}$.  Inclusion $S\hookrightarrow S_\cP$ induces a model metric $\sigma_m$ on $S$.  The collection of extended split level surfaces are \emph{monotonically arranged} in the sense that the elementary move $\cP_{k}\mapsto \cP_{k+1}$ corresponds to an isotopy $S_{\cP_{k}}$ to $S_{\cP_{k+1}}$ in the positive direction along $\bR$.  For details of the constructions in this paragraph, see \cite{minsky:ELTI} \S5 and \S8, as well as \cite{brock-canary-minsky:ELTII}, \S4.

\subsection{The model map}
Minsky builds a Lipschitz map $M_\nu\to M_\rho$, and Brock--Canary--Minsky promote this map into a bi-Lipschitz homeomorphism.  An application of  Sullivan Rigidity completes the proof of the Ending Lamination Theorem in the case of surface groups.   We quote here the properties of the bi-Lipschitz mapping that we use.
\begin{thm}[Bi-Lipschitz Model Theorem \cite{brock-canary-minsky:ELTII}, \cite{minsky:ELTI}]
There exists $K$, depending only on $S$, such that for any discrete, faithful Kleinian surface group $\rho: \pi_1(S) \to \G$ end invariants $\nu = (\nu_+, \nu_-)$, there is an orientation-preserving $K$-bi-Lipschitz homeomorphism
\[\Phi : M_\nu \to M_\rho\]
in the homotopy class determined by $\rho$.
\end{thm}

\subsection{Properties of extended split level surfaces}

We collect some facts about extended split level surfaces which follow directly from the construction of the model manifold and the Bi-Lipschitz Model Theorem.
\begin{prop}[Minsky \cite{minsky:ELTI}, Brock-Canary-Minsky \cite{brock-canary-minsky:ELTII}]\label{Split}
Let $\rho: \pi_1(S)\to \G$ be discrete and faithful with end invariants $\nu$.  Let $\cP$ be the base curves of a slice of some fixed resolution of $H_\nu$ and $\Phi:M_\nu\to M_\rho$ be a $K$-bi-Lipschitz model map.  Then there is an extended split level surface $S_\cP\subset M_\nu$ and an embedding $G: S\to M_\rho$ defined by the composition \[G:S\to S_\cP  \xrightarrow{\Phi}  M_\rho\] in the homotopy class determined by $\rho$.  Moreover, there are constants $A$, $L'>0$ depending only on $S$ such that
\begin{enumerate}[(i)]
\item $G$ induces a metric  $\sigma_G$ on $S$ such that $\area(S, \sigma_G)<A$.
\item For every $\alpha \in \cP$, the $\sigma_G$-length of $\alpha^*$ is no more than $L'$.
\item If $\cP'$ is the collection of base curves of a different slice and $\cP\cap \cP'=\emptyset$, then there is a map $G':S\to M_\rho$ also satisfying properties (i) and (ii) above, and $\im G\cup \im G'$ bounds an embedded submanifold $W$ in $M_\rho$ homeomorphic to $S\times I$.
\end{enumerate} 
\end{prop}
\begin{proof}
Since $\cP$ is the collection of base curves of a slice, there is an extended split level surface $S_\cP\subset M_\nu$.  Fixing a smooth structure on $S$, $S\to S_\cP$ is a piecewise smooth embedding and $\Phi$ is a $K$-bi-Lipschitz homeomorphism, so $G$ is an embedding.  By construction of $S_\cP$, the length of any boundary component of the standard pair of pants is $\epsilon_1$.  Indeed, the $\sigma_m$-length of the geodesic representative of any pants curve $\alpha\in \cP$ is bounded above by $\epsilon_1$, and so the $\sigma_G$-length of $\alpha^*$ is at most $K\epsilon_1 =L'$.  The $\sigma_m$-area of the annulus connecting boundary components of the standard pants is at most $4\epsilon_1$, because it is the union of two hyperbolic ruled annuli, which have area at most the sum of lengths of their boundary curves \cite{thurston:notes}.  Thus $S_\cP$ has area bounded above by $4(3g-3)\epsilon_1+ (2g-2)2\pi=A_0$ in the model metric, and so $\area(S, \sigma_G)\le K^2A_0=A$.  

Construct $G': S\to S_{\cP'} \to M_\rho$ analogously.  Each extended split level surface is isotopic to $S\times \{0\}$, and by monotonicity of the arrangement of extended split level surfaces, if $\cP$ and $\cP'$ share no common curves, then $S_\cP'\subset (S\times \bR )\setminus S_\cP$.  The image of an isotopy joining these surfaces defines an embedded submanifold $W'\subset M_\nu$.  Then $W=\Phi(W')\subset M_\rho$ is embedded and bound by $\im G\cup \im G'$.
\end{proof}

\section{Bounding volume from below}\label{below}

Let $\rho: \pi_1(S)\to \G$ be discrete and faithful without parabolics.  The goal of this section will be to show that a pair of simplicial hyperbolic surfaces that are far apart in $M_\rho$ are close to a pair of embedded surfaces which bound an embedded submanifold of large volume.  The main difficultly is to obtain uniform estimates for volume of the homotopy between our simplicial hyperbolic surfaces and our embedded surfaces, even as the diameter of our simplicial hyperbolic surfaces tends to infinity---a feature of manifolds with unbounded geometry.  This is the main technical point of this paper.  Here we adapt some of the ideas of Brock in \cite{brock:wp} by decomposing a homotopy space $S\times I$ into a union of solid tori.  We use a kind of isoperimetric inequality for simplicial $1$, $2$, and $3$ chains to bound the volume of our homotopy between an embedded surface and our simplicial hyperbolic surface.

Given a piecewise differentiable 3-chain $C\in \Cb_3(M_\rho)$, the function $\deg_C : M_\rho \to \bZ$ which measures the degree of $C$ in $M_\rho$ is well defined almost everywhere.  Moreover, if $P$ is a smooth 3-manifold, given a piecewise differentiable map $H: P \to M_\rho$, define the $H$-mass of $C\in \Cb_3(P)$ by \[\mass_H(C) = \int_{M_\rho}|\deg_{H_*C}|  \omega.\]
Notice that \[\left|\int_C H^*\omega\right|\le \mass_H(C).\]  For $Z\in \Cb_2(P)$ define \[\mass_H(Z)= \int_Z |\deg_{H_*Z}|d\area,\] where $\deg_{H_*Z}$ is a map from the space of 2-planes in the tangent bundle over each point in $M_\rho$ to $\bZ$, and $d\area$ is the Riemannian area form pulled back to each of the $2$-simplices of  $Z$.  We use the fact, due to Thurston (\cite{thurston:vol}, \S4) that if $\pi_1(P)$ is abelian and $\partial C=Z$, then \begin{equation}\label{mass}\mass_H(C)\le \mass_H(Z).
  \end{equation}  

Recall that if $(F,\cT)$ is a simplicial hyperbolic surface, then the metric $\sigma_F$ on $S$ induced by $F$ is one of constant negative curvature equal to $-1$ except perhaps at the vertices, where it has ``concentrated negative curvature.''  The metric $\sigma_F$ determines a conformal structure on $S$ and in this conformal class, there exists a unique hyperbolic metric $\sigma_F^{\hyp}$.  By a lemma of Ahlfors \cite{Ahlfors:lemma}, the identity mapping $(S,\sigma_F^{\hyp}) \to (S, \sigma_F)$ is 1-Lipschitz.  If $\sigma$ is a metric on $S$, and $\alpha$ is an arc or closed curve let $(\alpha^*,\sigma)\subset S$ be a $\sigma$-geodesic representative  of such a closed curve or arc, rel boundary.

\begin{lemma}\label{hmtpy} Let $\cP$ be the base curves of a slice of a fixed resolution of $H_\nu$, and $(F, \cT)$ be a simplicial hyperbolic surface adapted to $\cP$ in the homotopy class determined by $\rho$.  Then there is an extended split level surface $S_\cP\subset M_\nu$ such that the embedding $G: S\to S_\cP \to M_\rho$ is homotopic to $F$ by a homotopy $H$ and a triangulation $C$ of $S\times I$ such that \begin{equation}\label{expr}
\left|\int_{S\times I} H^*\omega\right| \le \mass_H(C)<M\end{equation} where $M$ depends only on $S$.  Moreover, $\partial H_*C = F_*\cT - G_*\cT$.  
\end{lemma}

\begin{proof}
Our plan is to start building  $H$ on disks $\beta\times I$ and annuli $\alpha\times I$ where $\alpha, \beta\subset S$.  These disks and annuli will decompose $S\times I$ into a union of solid tori $A\times I$ where $A\subset S$ is an annulus.  We then continue to extend $H$ on each solid torus $A\times I$.  We will triangulate $\partial (A\times I)$ and obtain upper bounds on the $H$-mass of the $2$-chains $Z_A$ corresponding to the triangulation.  By Inequality (\ref{mass}), for any $3$-chain $C_A$ with $\partial C_A = Z_A$, $\mass_H(C_A)\le \mass_H(Z_A)$, because $A\times I$ has abelian fundamental group.  Summing over solid tori appearing in our decomposition of $S\times I$ and ensuring that $Z_A$ and $Z_{A'}$ agree on their boundaries when $\partial A \cap \partial A' \not= \emptyset$ will yield (\ref{expr}).  Some additional bookkeeping will be required to ensure that   $\partial H_*C = F_*\cT - G_*\cT$.  We will proceed in several steps.  In each step, we will gather bounds on the $H$-mass of various chains. 

\begin{step}\label{step1}Define $H$ on $\cP$.
\end{step}

The curves of $\cP$ appear in the hierarchy $H_\nu$, so their length in $M_\rho$ is uniformly bounded above by a number $L$ which is at least the Bers constant for $S$ (\cite{minsky:ELTI}, Lemma 7.9).  Our map $G:S\to M_\rho$ satisfies properties (i)-(iii) as in Proposition \ref{Split} in the homotopy class of $F$.  The inclusion $S\to S_\cP$ induces a model metric $\sigma_m$ on $S$. 
By precomposing the map $G: S\to S_\cP\to M_\rho$ with an isotopy of $S$, we may assume that for each $\alpha\in \cP$, $(\alpha^*,\sigma_m) = (\alpha^*, \sigma_F)$.  Identify $\alpha\in \cP$ with its geodesic representative in these metrics, \ie $\alpha = (\alpha^*,\sigma_m) = (\alpha^*,\sigma_F)\subset S$. Note also that $G$ induces a metric $\sigma_G$ on $S$, and ${(S,\sigma_m)\to (S, \sigma_G)}$ is $K$-Lipschitz, because the model map is $K$-Lipschitz.


Let $\alpha\in \cP$, and consider the curves $F(\alpha)$, $G(\alpha)\subset M_\rho$.  They are homotopic by  ${h_\alpha:\alpha \times I \to M_\rho}$.  We may assume that for each $x\in \alpha$, $h_\alpha(x,I)$ is a geodesic segment.  Then the image of $h_\alpha$ is a ruled annulus, and so it's area is no more than the sum of the lengths of its boundary components.  Define $H:\cP\times I \to M_\rho$ so that it agrees with $h_\alpha$, for each $\alpha\in \cP$.  If $\tau_\alpha\in \Cb_2(\alpha \times I)$ is a triangulation of the annulus, then $\mass_H(\tau_\alpha)\le L+L'$.     
Our homotopy is now defined on $\cP\times I$, and every component of $S\times I\setminus \cP\times I$ is homeomorphic to the product of a pair of pants and the interval, which is a genus 2 handlebody.  

\begin{step}\label{step2} For each component $Y\subset S\setminus \cP$, decompose $Y$ into a union of annuli $Y = A_1 \cup A_2 \cup A_3$ such that the length of every boundary component of every annulus is uniformly bounded above in both the $\sigma_G$ and $\sigma_F$ metrics.  
\end{step}

Choose a component $Y\subset S \setminus \cP$ and two of three boundary curves $\alpha_1$, $\alpha_2\subset \partial Y =\alpha_1\cup\alpha_2\cup\alpha_3$.  For $i=1, 2$, find collar neighborhoods $A_i=\collar(\alpha_i)\cap Y$  with respect to the $\sigma_F^{\hyp}$ metric.  By the Collaring Theorem, we may choose $A_i$ such that if $a_i =\partial A_i\cap \interior Y$, then $2L>\ell_{\sigma_F^{\hyp}}(a_i ) > 2$.  Since the identity $(S, \sigma_F^{hyp})\to (S, \sigma_F)$ is $1$-Lipschitz, $\ell_{\sigma_F}(a_i)<2L$.  Take $A_i' = Y\cap U(\alpha_i)$ and $a_i'=\partial A_i'\cap \interior Y$, where we have included $Y$ into $S_\cP$, so that by construction of extended split level surfaces, $\ell_{\sigma_m}(a_i')=\epsilon_1$.  Then $\ell_{\sigma_G}(a_i')\le K\epsilon_1=L'$ (which is actually the constant from Proposition \ref{Split} Property (ii)), because $(S, \sigma_m)\to (S, \sigma_G)$ is $K$-Lipschitz.  By construction, $F$  maps $\alpha_1$ and $\alpha_2$ to their geodesic representatives in $M_\rho$, so $\ell_{\sigma_F}(\alpha_i)\le L$ and $\ell_{\sigma_G}(\alpha_i)\le L'$.  We have now identified $A_1$ and $A_2$, and verified that their boundary curves have bounded lengths.  Our next goal is to identify $A_3$.

\begin{figure}[h]
\includegraphics[width=4in]{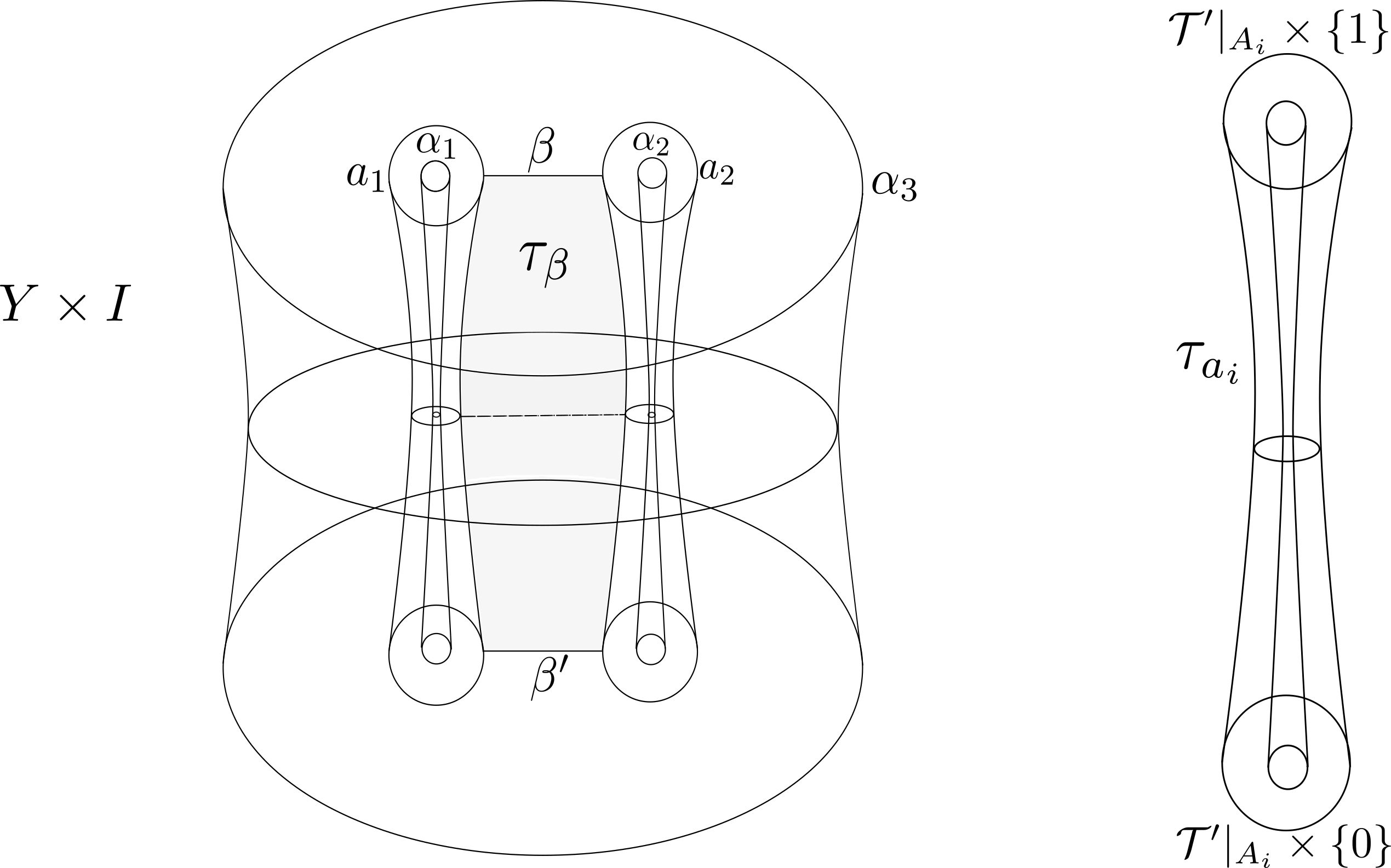}  
\caption{\small{Triangulating the picture on the right yields a $2$-chain 
$Z_i = \tau_{\alpha_i}+ \cT'|_{A_i}\times\{0\}+ \tau_{a_i}+ \cT'|_{A_i}\times\{1\}$
that represents the torus $\partial(A_i\times I)$ for $i=1$, $2$.  Not labeled on the right is $\tau_{\alpha_i}$, which is a triangulation of annulus $\alpha_i\times I$.}}
\label{fchains}
\end{figure}

There is a $\sigma_F^{hyp}$ geodesic arc $\beta$ meeting $a_1$ and $a_2$ orthogonally each in one point, $v_1$ and $v_2$ respectively.  By an argument evoking the hyperbolic trigonometry of a right angled hexagon, there is a constant $b$ such that $\ell_{\sigma_F^{\hyp}}(\beta)<b$.  Similarly, there is a $\sigma_m$ geodesic arc $\beta'$ meeting $a_1'$ and $a_2'$ orthogonally together with a constant $b'$ so that $\ell_{\sigma_m}(\beta')<b'/K$.  Since $(S, \sigma_F^{\hyp})\to (S,\sigma_F)$ is 1-Lipschitz, $\ell_{\sigma_F}(\beta) \le b$.  Similarly, $\ell_{\sigma_G}(\beta')\le b'$.  As in Step \ref{step1}, by precomposing $G$ with a homeomorphism of $S$ isotopic to the identity fixing $\cP$ pointwise, we may assume that $a_i = a_i'\subset Y$ for $i=1$, $2$ and that $\beta=(\beta^*,\sigma_F^{hyp}) = ({\beta'}^*,\sigma_m)\subset Y$.    The surface $Y\setminus (a_1 \cup \beta \cup a_2)$ decomposes as a disjoint union of (the interiors of) the two collars $A_1$ and $A_2$ together with a third annulus $A_3$ with core curve homotopic to $\alpha_3$ (see Figure \ref{fchains}). 
  Then $\partial A_3$ has two components $\alpha_3$ and $a_3$.  We have estimates $\ell_{\sigma_F}(a_3) \le 2b + 4L$ and $\ell_{\sigma_G}(a_3)\le 2b' + 2L'$.  Again, $\ell_{\sigma_F}(\alpha_3) \le L$ and $\ell_{\sigma_G}(\alpha_3)\le L'$, which means that we have completed Step \ref{step2}.  

\begin{step}\label{step3} Define $H$ on  $Y\times I$, extending $H|_{\partial Y}$.  
\end{step}

We  join the arcs $F(\beta)$ and $G(\beta)$ with geodesic segments to obtain a ruled disk $D_\beta: \beta \times I \to M_\rho$ with triangulation $\tau_\beta$.  Join $F(a_1\setminus\{v_1\})$ to $G(a_1\setminus\{v_1\})$ and $F(a_2\setminus\{v_2\})$ to $G(a_2\setminus\{v_2\})$ by geodesic ruled disks $h_{a_i}: a_i\setminus\{v_i\} \times I \to M_\rho$ with triangulations $\tau_{a_i}$.  These ruled disks extend $H$ on all of $(a_1\cup\beta\cup a_2)\times I$.  The area of a geodesic ruled annulus or disk is at most the length of the boundary components that define the ruling.  Using our length estimates from Step \ref{step2}, we see that  $\mass_H(\tau_\beta)\le b+b'$ and  $\mass_H(\tau_{a_i})<2L +L'$ for $i = 1,2$.

Subdivide $\cT$ to a triangulation $\cT'$ that includes $\beta$, $a_1$, and $a_2$, where each is subdivided at its intersection with an arc of $\cT$.  Define $H_0$ on the triangles of $\cT'$ by the restriction of $G$, and define $H_1$ on the triangles of $\cT'$ by the restriction of $F$.   $H$ is now defined on a union of tori $\partial(A_i\times I)$ glued along their boundaries.  The image of each torus bounds an immersed solid torus.  Extend $H$ accordingly so that it is piecewise smooth and defined everywhere on $Y\times I$.   
\begin{step}\label{step4}
Bound the $H$-mass of triangulations of each solid torus $A_i\times I$, for $i =1, 2, 3$.
\end{step}

The torus $Z_3 = \partial (A_3\times I)$ decomposes as a union of 4 annuli as follows.  In the $\sigma_F$-metric, $A_3$ is a proper subset of the 1-Lipschitz image of a hyperbolic pair of pants, so $\mass_H(\cT'|_{A_3}\times \{1\})\le 2\pi$.  The `outer' annulus $\tau_{\alpha_3}$ has $H$-mass at most $L+L'$, and the `inner' annulus is a sum $\tau_{a_1} + \tau_\beta+\tau_{a_2}+\tau_\beta$, so it's $H$-mass is at most $2(2L+L') + 2(b+b')$.  Finally, the `bottom' annulus $(A_3,\sigma_G)$ is a proper subset of the $K$-Lipschitz image of a hyperbolic pair of pants, so  $\mass_H(\cT'|_{A_3}\times\{0\})\le K^22\pi$.  Now, $Z_3$ bounds a 3-chain $C_3$ which is a triangulation of the solid torus $A_3\times I$.  Lifting to the cover $\bH^3/\langle\rho(\alpha_3)\rangle$, we can apply Inequality (\ref{mass}) and conclude that  \[\mass_H(C_3)\le \mass_H(Z_3)\le 2\pi + 3(L+L') + 2(b+b')+ K^22\pi=M'.\]    

We can express the torus $Z_1 = \partial (A_1\times I)$ as the sum \[Z_1=\tau_{\alpha_1} + \cT'|_{A_1}\times \{1\} + \tau_{a_1} + \cT'|_{A_1}\times \{0\}.\]  Being geodesic ruled annuli,  $\area (A_1,\sigma_{F})\le 3L$ and $\area(A_1,\sigma_G)\le 2L'$, so   \[\mass_H(Z_1)\le (L +L') + 3L + (2L +L') + 2L' =M''.\]  As before, $Z_1$ bounds a chain $C_1$ and $\mass_H(C_1)\le M''$.  We have an entirely analogous estimate for the $H$-mass of the 3-chain $C_2$ bound by $Z_2$, which represents $ \partial(A_2\times I)$.  

\begin{step}\label{step5} Final step.  
\end{step}

There were $2g-2$ components of $S\setminus \cP$ and we constructed $3$-chains for three solid tori on each component. The sum of these chains $C$ defines a $3$-chain satisfying $\partial C = \cT'\times \{1\} - \cT'\times \{0\}$.  Let ${M''' = \max\{M', M''\}}$.  We have shown that \[\left|\int_{S\times I} H^*\omega\right| \le \mass_H(C)\le3(2g-2)M''' = M. \]

In step \ref{step3}, we arranged so that we did not create any new vertices on $\cT$ which were not already part of an edge of $\cT$.  Thus, when we subdivide $\cT$ to obtain $\cT'$ to include the new edges, $(F,\cT')$ is still a simplicial hyperbolic surface with the same image as $(F,\cT)$, and $\cT'$ is just a refinement of $\cT$.  We may homotope the new vertices of $F\cT'$, along edges of $F\cT'$, back onto vertices of $F\cT$.  Since $\im F$ is piecewise geodesically convex, the straightened homotopy has image contained in $\im F$, so its volume is zero.  We may therefore assume that $H$ preserves the triangulation $\cT$ and $\partial HC = F\cT - G\cT$. 
\end{proof}

The following construction is a summary of what we have done so far.  Later in this section, we will show that $|\hat\omega(V_k)|\to \infty$ as $k\to \infty$, which is one of two main steps in the proof of Theorem \ref{main1}.

\begin{construction}\label{chains}
Suppose $\rho : \pi_1(S)\to \G$ is discrete and faithful with end invariants $\nu= (\nu^-, \nu^+)$.  Assume $\nu^+\in \partial\cC(S)=\cEL(S)$.

\begin{enumerate}
\item Find a hierarchy of tight geodesics $H_\nu$ joining these end invariants.  
\item Fix a resolution of $H_\nu$ and recover markings $\{\mu_k\}_{k\in J}$ with base curves $\cP_k= \mu_k$, where $J\subset \bZ$ is a (bi)-infinite interval. 
\item For each $k$, find a simplicial hyperbolic surface $F_k:S \to E_S$ whose associated triangulation, $\cT_k$ is adapted to $\cP_k$.
\end{enumerate}
Then for each $k$, there is a 3-chain $V_k$ such that $\partial V_k = F_k\cT_k - F_0\cT_0$.   
\end{construction}

\begin{proof}
Since each of the $F_k$ is a homotopy equivalence, and the $\cT_k$ are all generators for $\Hb_2(S; \bZ)$,  we have that $[F_k\cT_k - F_0\cT_0]=0 \in \Hb_2(S\times\bR;\bZ)$.  There is then a 3-chain $V_k$ such that $\partial V_k = F_k\cT_k - F_0\cT_0$.   
\end{proof}

Now we will prove that the chains $V_k$ from Construction \ref{chains} have large volume after being straightened.  The idea is to use the extended split level surfaces coming from the model $M_\nu$ to find embedded submanifolds of $M_\rho$ with large volume that are `trapped' by our simplicial hyperbolic surfaces.  We need to be able to control the volume of the homotopy between simplicial hyperbolic surfaces and the boundary of our embedded submanifolds, hence Lemma \ref{hmtpy}.  In \S7, we will refine these estimates.  We include this case here, because the work done in \S7 is somewhat more involved.  It uses more extensively the hierarchy machinery and structure of the model.  Here we use only the existence and geometry of extended split level surfaces.  

\begin{lemma}\label{lowerbound}
With $V_k$ as in Construction \ref{chains}, given $n>0$ there is a $k_0$ such that for all $k>k_0$, 
 \[|\hat\omega(V_k)| > n.\]
\end{lemma}

\begin{proof} For $k$ large, we have that $\cP_0\cap\cP_k=\emptyset$.  Then we can find extended split level surfaces $S_{\cP_0}$ and $S_{\cP_k}$ and $G_0$ and $G_k$ with disjoint images bounding a submanifold $W_k\subset M_\rho$ homeomorphic to $S\times I$ as in Proposition \ref{Split}.  By Lemma \ref{hmtpy}, there are homotopies   \[H_0,~H_k : S\times I \to M_\rho\] and triangulations $C_0$ and $C_k$ of $S\times I$ such that \[\left|\int_{S\times I} H_j^*\omega\right|\le\mass_{H_j}(C_j)<M\] for $j=0, k$.     Triangulate $W_k$ so that $\partial W_k = G_k\cT_k - G_0\cT_0$.  Then \[\str \partial V_k =\str(  F_k\cT_k-F_0\cT_0) =  F_k\cT_k-F_0\cT_0= \partial ( C_0+W_k+ C_k).\]  

What follows is our basic estimate for $|\hat\omega(V_n)|$ from below. 
Since $M_\rho$ is not compact, $\Hb^3_{\dR}(M_\rho;\bR) = 0$ and so there is a smooth 2-form $\zeta\in \Omega^2(M_\rho)$ such that $d\zeta = \omega$.   By Stokes' Theorem, and since $\str$ is a chain map, 

\begin{align*}
|\hat\omega(V_k)| &= \left|\int_{\str V_k}\omega\right| = \left|\int_{\str V_k} d \zeta\right| = \left|\int_{\partial \str V_k} \zeta\right|
\\ & = \left|\int_{\str{\partial V_k}}\zeta\right| = \left|\int_{ \partial ( C_0+ W_k+ C_k) }\zeta\right| 
\\ &\ge \left|\int_{W_k}\omega\right| - \left|\int_{S\times I}H_0^*\omega\right| - \left|\int_{ S\times I} H_k^*\omega\right|
\\ & \ge \vol(W_k) - 2M.
\end{align*}

By construction, $\{W_{k}\}_{k\in \bZ_{\ge0}}$ is a compact exhaustion of the closure of $E_S\setminus K_0$, where $K_0$ is the pre-compact component of $E_S\setminus W_0$.  If the boundary surfaces $\partial W_k$ stay in a compact set, then since they each contain a closed curve of bounded length, we could extract a closed curve of bounded length which is a limit of other bounded length curves, violating discreteness of $\im\rho$.  So the function $k \mapsto \vol(W_{k})$ is increasing and proper. Hence we may find a $k_0$ such that for $k>k_0$, $\vol(W_k)-2M > n$.
\end{proof}

\section{Bounding volume from above}\label{sectionabove}
Using the smooth structure of $M_\rho$, we can bound the size $|\hat\omega(V_i)|$ of any collection $\{V_i\}$ of 3-chains whose boundaries are contained in a compact set $K$.  For an integral chain $V$, we will often abbreviate $\|V\|_1$ to $|V|$.  Our bound will depend on $K$ as well as $|\partial V_i|$.  

As was observed earlier, there is a smooth 2-form $\zeta\in \Omega^2(M_\rho)$ such that $d\zeta = \omega$.
Define $\hat \zeta : \Cb_2(M_\rho) \to \bR$ by $\tau\mapsto \int_ {\str\tau} \zeta$.  Since $\str$ is a chain map, by Stokes' Theorem we have
\[\delta\hat\zeta(\sigma) = \int_{\str\partial\sigma}\zeta =\int_{\partial \str \sigma} \zeta= \int_{\str \sigma}d\zeta =  \int_{\str\sigma} \omega= \hat\omega(\sigma). \] 

Let $K\subset M_\rho$ be a co-dimension 0, compact submanifold.   Then $\zeta$ defines a smooth, antisymmetric bilinear function $\zeta(K)$ on the space of $2$-planes over each point in $K$.  This space is compact, so this function achieves its maximum, called $\|\zeta(K)\|_\infty<\infty$.

\begin{lemma}\label{upper}
 Let $K\subset M_\rho$ be a compact set, and suppose that $(F_0,\cT_0)$ and $(F_1,\cT_1)$ are simplicial hyperbolic surfaces that have nonempty intersection with $K$.  Suppose  $V\in \Cb_3(M_\rho)$ is such that $\partial V = F_1\cT_1 -F_0\cT_0$.  There is a constant $m = m(K, M_\rho)$ such that \[\hat\omega(V) \le m|\partial V|.\]
\end{lemma}

\begin{proof}

We may assume that $K$ is a co-dimension 0 submanifold of $M_\rho$.  
If $\tau : \Delta_2\to K$ is a singular 2-simplex, 
 \begin{equation}\label{areaD}
 |\hat\zeta(\tau)|=\left|\int_{\str\tau} \zeta\right|\le \|\zeta(K)\|_\infty\int_{\str\tau} d\area \le \|\zeta(K)\|_\infty\pi. 
 \end{equation}

In (\ref{areaD}), $d\area$ is the hyperbolic area, and we have used the fact that any hyperbolic triangle has area no more than $\pi$.  

By the Bounded Diameter Lemma, there is a compact $K'\supset K$ such that if  $(F,\cT)$ is a simplicial hyperbolic surface whose intersection with $K$ is non-empty, then $\im F\subset K'$.  Take $m = \pi \|\zeta(K')\|_\infty$, so that  \[|\hat\omega(V)| = |\hat\zeta(\partial V)| \le \sum_{\tau \in \partial V} |\hat\zeta(\tau)| \le \sum_{\tau \in \partial V} \pi\|\zeta(K')\|_\infty= m|\partial V|. \]
\end{proof}
The point of this section is to provide the following uniform upper bound.
\begin{lemma}\label{uniformupperbound}
Suppose that $\lambda\subset S$ is a lamination which is not an ending lamination for $M_\rho$, and let $\alpha_0,$ $\alpha_1$, ... be simple closed curves on $S$ whose Hausdorff limit contains $\lambda$.   Find a sequence of simplicial hyperbolic surfaces $(F_i, \cT_i)$ adapted to pants decompositions containing $\alpha_i$.  Suppose that $V_1$, $V_2$, ... is a sequence of $3$-chains satisfying $\partial V_i = F_i\cT_i-F_0\cT_0$. Then there is a number $C$ such that \[|\hat\omega(V_i)| \le C\] for all $i$.  
\end{lemma}
\begin{proof}
Since $\lambda$ is not an ending lamination for $M_\rho$, the $\alpha_i^*$ stay in a compact set $K$, which by construction, all of the $\im F_i$ intersect.  By the Bounded Diameter Lemma, we find a compact set $K'$ such that $\im F_i\subset K'$ for all $i$.  We apply Lemma \ref{upper} to find an $m = m(K')$ such that $|\hat\omega(V_i)| \le |\partial V_i| m= 2(10g-10)m$ for all $i$.  
\end{proof}

\section{Infinite ends are cohomologically distinct}\label{main1}
We are now ready to state the first of our main results.  We first consider closed surface groups, and generalize to the setting of finitely generated Kleinian groups without parabolic cusps by lifting to the covers corresponding to the ends of quotient manifolds.  

We begin with the following observation.  The following lemma will allow us to replace a map $S\to S\times \bR$ with a homotopic map at a bounded cost. 
\begin{lemma}\label{straighten}Let $X$ be a connected CW complex and $\alpha\in \Cb_b^{n+1}(X;\bR)$.  Suppose $f, g: Y\to X$ are homotopic mappings of an $n$-dimensional triangulated CW complex $Y$.  There is an $(n+1)$-chain $U$ such that $\partial U = fY-gY$ and such that \[|\alpha(U)| \le (n+1)|Y| \|\alpha\|_\infty.\]  
\end{lemma}

\begin{proof}
Find a homotopy $H : Y\times I \to X$ so that $H_0 = g$ and $H_1 = f$. The triangulation of $Y$ gives a cell decomposition of $Y\times I $ by prisms.  Apply the prism operator to obtain $u \in \Cb_{n+1}(Y\times I)$ with the property that $\partial u =fY-gY$.  The number of $n$-faces of $Y$ is $|Y|$, and so $|u| = (n+1)|Y|$ because each prism decomposes as a union of $n+1$ $(n+1)$-simplices.  Take  $ U =Hu$ and write  $U= \sum_{i = 1}^{|U|} \sigma_i$.  Then we have 

\[|\alpha(U)| \le \sum_{i=1}^{|U|} | \alpha(\sigma_i) | \le (n+1)|Y|\|\alpha\|_\infty.\]

\end{proof}
To apply this observation, suppose that $\rho: \pi_1(S)\to \G$ is a marked Kleinian surface group without parabolics, $(F,\cT)$ is a simplicial hyperbolic surface inducing $\rho$ on fundamental groups with triangulation adapted to a pants decomposition $\cP$, and $G: S\to M_\rho$ is any other map homotopic to $F$.  If $U$ is a $3$-chain with boundary $\partial U = F\cT-G\cT$, applying Lemma \ref{straighten}, we have 
\begin{equation}\label{application}
|\hat\omega(U)| \le 3(10g-10)v_3.
\end{equation}

The objective here is to compare the bounded fundamental classes of  distinct hyperbolic metrics on $S\times \bR$; we need a chain complex where we can compare them.  We will work in the space $\Cb_b^3(S\times \bR;\bR)$.  Recall that $\Hyp_0(S\times \bR)$ is the set of pairs $(f,N)$, where $f: S\times \bR \to N$ is a homotopy equivalence and $N$ is a hyperbolic manifold homeomorphic to $S\times \bR$ without parabolic cusps.  Let $\{(f_\alpha, N_\alpha):\alpha\in \Lambda \} \subset \Hyp_0(S\times \bR)$.  We will give an alternate description of the bounded fundamental class so that we may compare the co-cycles arising from different metrics on $S\times \bR$ directly.  
There are diffeomorphisms in the homotopy classes of the $f_\alpha: S\times \bR \to N_\alpha$.  Assume $f_\alpha$ are as such.  Pull back the hyperbolic metrics via $f_\alpha$ to obtain hyperbolic metrics on $S\times \bR$ called $\sigma_\alpha$.  Define straightening maps by conjugation so that the diagram \[\xymatrix{ 
{\Cb_\bullet (S\times \bR)} \ar[r]^{\str_\alpha} \ar[d]_{{f_\alpha}_*} &{\Cb_\bullet (S\times \bR)}  \\
{\Cb_\bullet (N_\alpha)} \ar[r]^{\str} & {\Cb_\bullet (N_\alpha)} \ar[u]_{{f_\alpha\inverse}_*}
}\]
commutes.  Then each $\str_\alpha$ is a chain map, because it is a composition of chain maps.   Push forward the volume form $\omega\in \Omega^3(\bH^3)$ to obtain volume forms $\omega_\alpha\in \Omega^3(S\times \bR)$ using the composition \[ \xymatrix{
 & {\bH^3} \ar[d]_{\pi_\alpha} \ar[dl]\\
S\times \bR &\ar[l]^{f_\alpha\inverse}   {N_\alpha}}
\]
so that $(f_\alpha\inverse\circ \pi_\alpha)^*\omega_\alpha = \omega$ for $\alpha\in\Lambda$.  Define $\hat\omega_\alpha: \Cb_3(S\times \bR) \to \bR$ by $\hat\omega_\alpha( \sigma) = \int_{\str_\alpha \sigma}\omega_\alpha$.  Recall that $\Hb^3_b(\pi_1(S);\bR)$ is isometrically identified with $\Hb^3_b(S\times \bR; \bR)$.  

The main idea of the proof of the following theorem is that we can find chains in $\Cb_3(S\times \bR)$ whose boundaries have uniformly bounded complexity, and such that the volume of this chain is very large when straightened with one metric while uniformly bounded with respect to another straightening.  Thus, if we write any finite linear combination of co-cycles $\sum a_i \hat\omega_{\alpha_i}$ as a co-boundary $B$, we see that $B$ cannot be a bounded co-boundary when the end invariants of $\{N_{\alpha_i}\}$ are sufficiently different.  
\begin{theorem}\label{mainthm1}
Let $\{(f_\alpha, N_\alpha):\alpha\in \Lambda \} \subset \Hyp_0(S\times \bR)$ be such that $N_\alpha$ has a geometrically infinite end invariant $\lambda_\alpha\in \cEL(S)$ that is different from the geometrically infinite end invariants of $N_\beta$ whenever $\alpha\not=\beta\in \Lambda$.    Then \[\{[\hat\omega_\alpha]: \alpha\in \Lambda\} \subset \Hb^3_b(\pi_1(S);\bR)\] is a linearly independent set.    
\end{theorem}

\begin{proof} To  show linear independence, we consider finite linear combinations $\sum_{\ell = 1}^n a_\ell[\hat\omega_{\alpha_\ell}]$. Suppose that  $\sum_{\ell = 1}^n a_\ell[\hat\omega_{\alpha_\ell}] = 0$, \ie suppose that there is an $A\in \Cb^2_b(S\times \bR; \bR)$ such that $\delta A = \sum_{\ell = 1}^n a_\ell\hat\omega_{\alpha_\ell}$.   Fix $i \in  \{ 1, ..., n\}$, and without loss of generality, assume that $\nu_{\alpha_i}^+ = \lambda_{\alpha_i}$. Apply Construction \ref{chains} to the representation ${f_{\alpha_i}}_*: \pi_1(S)\to \G$ to obtain a sequence of pants decompositions $\cP_k^i$, simplicial hyperbolic surfaces $(F_k^i:S\to N_{\alpha_i},\cT_k^i)$, and 3-chains $V_k^i$.  We identify $N_{\alpha_i}$ isometrically with $(S\times \bR, \sigma_{\alpha_i})$ and assume that $\im F_k^i\subset S\times \bR$, for all $k$.  

For $j\not = i$, we now find simplicial hyperbolic surfaces $B_k^j: S\to N_{\alpha_j}= (S\times \bR, \sigma_{\alpha_j})$ adapted to $\cT_k^i$.  For fixed $j$, the surfaces $\im B_k^j\subset (S\times \bR, \sigma_{\alpha_j})$ cannot leave every compact set. Otherwise, since the $B_k^j$ map the curves $\alpha_k^i\in \cP_k^i$ to their geodesic representatives in the $\sigma_{\alpha_j}$ metric, we would have that $\alpha_k^i\to \nu_{\alpha_j}^+$ or $\alpha_k^i\to\nu_{\alpha_j}^-$ as $k\to \infty$ by definition of the end invariants.  Since $\alpha_k^i\to \lambda_{\alpha_i}$ as $k\to \infty$, this would contradict our hypothesis that $\nu_{\alpha_j}^+\not=\lambda_{\alpha_i}$ and $\nu_{\alpha_j}^-\not=\lambda_{\alpha_i}$ for $i\not=j$.  Construct 3-chains $C_k^j$ with boundary $B_k^j\cT_k^i - B_0^j\cT_0^i$. Then by Lemma \ref{uniformupperbound}, \begin{equation}\label{high}|\hat\omega_{\alpha_j}(V_k^i)| = |\hat\omega_{\alpha_j}(V_k^i -C_k^j)+ \hat\omega_{\alpha_j}(C_k^j)| \le |\hat\omega_{\alpha_j}(V_k^i -C_k^j)| +C_{i,j} \end{equation} where $C_{i,j}$ depends on the sequence $ \{\cP_k^i\}$ and the metric $\sigma_{\alpha_j}$, but is independent of $k$.  
By Lemma \ref{straighten}, \begin{equation}\label{app2}|\hat\omega_{\alpha_j}(V_k^i-C_k^j)| \le 2c\end{equation} where $c=v_3(10g-10)3$; compare this with the discussion preceding Inequality (\ref{application}) above.  Combining (\ref{high}) and (\ref{app2}) and collecting constants, we have \begin{equation}\label{UPPER} |\hat\omega_{\alpha_j}(V_k^i)|\le c_{i,j}.
\end{equation}

Recalling that $\delta A = \sum_{\ell = 1}^n a_\ell \hat\omega_{\alpha_\ell}$, we have 
\begin{align*}
 |(\sum_{\ell = 1}^n a_\ell\hat\omega_{\alpha_\ell})(V_k^i) |=|(\delta A)(V_k^i)|=|A(\partial V_k^i)|\le\|A\|_\infty |\partial V_k^i| 
 \end {align*}
On the other hand, by Lemma \ref{lowerbound}, given $n$, we can find $k_0$ such that for $k>k_0$, $ |\hat\omega_{\alpha_i}(V_k^i)|>n$.  Combining this with (\ref{UPPER}), we have
\[ |( \sum_{\ell = 1}^n a_\ell\hat\omega_{\alpha_\ell})(V_k^i) |\ge  |a_i||\hat \omega_{\alpha_i}(V_k^i)| - \sum_{j\not= i}| a_j\hat\omega_{\alpha_j}(V_k^i)|> |a_i| n-\sum_{j\not = i} |a_j| c_{i,j},\]
whenever $k>k_0$.

Recalling that we chose $\cT_k^i$ so that for all $k$ we have $|\partial V_k^i|= |\cT_k^i|+|\cT_0^i| = 10g-10$, it follows that \begin{equation}\label{contradiction} \|A\|_\infty \ge \frac{|a_i| n - \sum_{j\not = i} |a_j| c_{i,j}}{10g-10} = \frac{|a_i|}{10g-10}n - c_i.\end{equation}
Since $n$ was arbitrary, and $\|A\|_\infty < \infty$, this means that $a_i = 0$. Now, since $i$ was arbitrary, it follows that $a_i = 0$ for all $i = 1, ..., n$.  This establishes linear independence of the set $\{[\hat\omega_\alpha]: \alpha\in \Lambda\}$ and proves the theorem.  

\end{proof}

Here we give an example of a collection that satisfies the hypotheses of Theorem \ref{mainthm1}.

\begin{cor}\label{singleboundary}
Let $S$ be a closed, orientable surface.  There is an injective map \[\Psi:  \cEL(S) \to \Hb^3_b(\pi_1(S);\bR)\] whose image is a linearly independent set.  
\end{cor}
\begin{proof}
Fix a point $X\in \teich (S)$.  By Thurston's Double Limit Theorem \cite{thurston:double}, for every $\lambda\in \cEL(S)$ there is a marked hyperbolic 3-manifold $(f_\lambda, N_\lambda)\in \Hyp_0(S\times \bR)$ with end invariants $(X,\lambda)$. Define $\Psi(\lambda)=[\hat\omega_\lambda] $.  The set of marked manifolds $\{(f_\lambda,N_\lambda): \lambda\in \cEL(S)\}$ satisfies the hypotheses of Theorem \ref{mainthm1}, so linear independence follows.  Injectivity follows from linear independence.  
\end{proof}

Let $\g\le \G$ be a non-elementary, finitely generated Kleinian group that contains no parabolic or elliptic elements such that $M_\g = \bH^3/\g$ has infinite volume.  We will again use techniques and tools from the proof of the Ending Lamination Theorem for parabolic-free hyperbolic manifolds with finitely generated fundamental group to obtain immediate generalizations of our results.  The key is to use Tameness to find a geometrically infinite end $E_S$ diffeomorphic to $S\times[0,\infty)$ and described by a model manifold, so that $S\to S\times [0,\infty)$ is incompressible, even if $S\to M_\g$ is compressible.    We then find a 3-chain with small boundary and large volume when straightened in one manifold but which has small straightened volume with respect to a second metric with a different end invariant. This is the same strategy as in the proof of Theorem \ref{mainthm1}.  Recall that in \S\ref{infiniteends}, we defined $\cEL(S,M_\g)$ to be the collection of equivalence classes of minimal, filling Masur domain laminations up to the action of $\Mod_0(S,M_\g)$.  If $S$ is incompressible, $\cEL(S,M_\g) = \cEL(S)$.  
  
We will consider collections of finitely generated Kleinian groups whose geometrically infinite ends are sufficiently different.
Suppose $J$ is an index set and $\{(f_i,N_i): i\in J\}\subset \Hyp_0(M_\g)$.  For each $i\in J$,  $f_i:M_\g \to N_i$ induces a marking \[{f_i}_*: \g\to\pi_1(N_i)=\g_i\le \G.\] Suppose $S_i\in \cE(\g_i)$ is a geometrically infinite end of $N_i$.  The inclusion $S_i\to N_i$ induces a homomorphism $\pi_1(S_i) \to \g_i$ with image called $\g_i(S_i)$; let $\g_j(S_i)$ denote the subgroup of $\g_j$ given by $ {{f_j}_*\circ {f_i}_*\inverse}(\g_i(S_i))$.  The surface $S_i$ lifts homeomorphically to the total space of the cover $M_{\g_i(S_i)} \to N_i$ and defines a geometrically infinite end with end invariant $\nu(S_i)$. By the Covering Theorem \cite{canary:covering}, it has no \emph{new} geometrically infinite ends.  That is, if $\lambda$ is an end invariant for a geometrically infinite end $\tilde{S}\in \cE(\g_i(S_i))$, then $\lambda$ is naturally an end invariant for $\pi(\tilde{S})\in \cE(\g_i)$. Moreover, the covering map induces a bijection $\cEL({S}_i, M_{\g_i(S_i)})\to \cEL(S_i,N_i)$.  We will consider collections of marked manifolds $\{(f_i,N_i): i \in J\} \subset \Hyp_0(M_\g)$ satisfying the following property:
\begin{enumerate}[(*)]
\item For each $i\in J$, there is a geometrically infinite end $S_i\in \cE(N_i)$ so that $\nu(S_i)\in \cEL(S_i, N_i)$ and $\nu(S_i)$ is different from every end invariant of $M_{\g_j(S_i)}$, as long as $i\not=j \in J$.

\end{enumerate}

In the following corollary, we lift to the covering spaces of  geometrically infinite ends and distinguish bounded classes in the cover.  The discussion would be slightly less technical and slightly less general if we chose not to lift to covers, but instead worked just in the neighborhoods of geometrically infinite ends.  Lifting to covers allows us to distinguish between the same ending laminations that appear on homotopically distinct boundary components of two manifolds.  
\begin{cor}\label{general}
Suppose $\g\in \G$ is finitely generated, Kleinian, torsion and parabolic free.  If $\{(f_i, N_i):{i\in J}\} \subset \Hyp_0(M_\g)$ is some collection of marked hyperbolic manifolds satisfying (*), then $\{[\hat\omega_i]: {i \in J}\} \subset \Hb_b^3(\g;\bR)$ is a linearly independent set.
\end{cor}
\begin{proof}
Let $S_i\in \cE(\g_i)$ be the geometrically infinite end of $N_i$ from our hypotheses.  Assume that $S_i$ is incompressible.   That is, the inclusion $S_i\to N_i$ induces an injection $\pi_1(S_i) \to \g_i$ with image $\g_i(S_i)$.  Then $\g_i(S_i)$ is a Kleinian surface group, and the total space of the  covering $M_{\g_i(S_i)} \to N_i$ has a geometrically infinite end with end invariant $\nu(S_i)$.  
By the Covering Theorem, if $\g_i(S_i)$ has infinite index in $\g_i$, then $M_{\g_i(S_i)}$ has exactly one geometrically infinite end.   If $\g_i(S_i)$ has finite index in $\g_i$, then $\g_i(S_i)=\g_i$.  By assumption, for all $j\in J$ with $i\not=j$, we have that $\nu(S_i)$ is different from the end invariants of $M_{\g_j(S_i)}$, so we can apply Theorem \ref{mainthm1}.   The inclusion $\g_i(S_i)\to \g_i\cong \g$ induces a linear mapping $\Hb^3_b(\g)\to \Hb_b^3(\pi_1(S_i))$.    The images of our classes differ in $\Hb_b^3(\pi_1(S_i))$, so they must differ in $\Hb_b^3(\g)$ as well.  This completes the proof of the corollary in the case that $S_i\to N_i$ is incompressible.
Assume now that $S_i\to N_i$ is compressible.  First, we will lift to the cover corresponding to $\g_i(S_i)$.   To ease some of the burden of notation, we do not distinguish between $\hat\omega_j$ and its image in $\Cb_b^3(M_{\g_j(S_i)};\bR)$ under the map induced by the covering $M_{\g_j(S_i)}\to N_j$. 

In this paragraph, we build a sequence of chains $V_k^i$ such that $|\hat\omega_i(V_k^i)|\to \infty$ as $k\to \infty$, which is the first estimate that we will need to prove the corollary.   There is a neighborhood $E_{S_i}\subset M_{\g_i(S_i)}$ of $S_i$ whose geometry is described by a model manifold, just as in the case of surface groups (see \cite{brock-canary-minsky:ELTII} \S1.3, and \cite{brock-canary-minsky:ELTIII}). Here, we describe how to build the model.  Find a pleated surface $f:S_i\to E_{S_i}$ in the homotopy class of the inclusion $S_i\to E_{S_i}$.  The induced metric on this pleated surface is hyperbolic, and so it defines a point in Teichm\"uler space $X_i\in\teich(S_i)$.  As in \S\ref{invnts}, $X_i$ determines a complete marking $\mu_{X_i}$, and we build a complete hierarchy $H_i$ with $\mathbf{I}(H_i) = \mu_{X_i}$ and $\mathbf{T}(H_i)=\nu(S_i)$.   Fix a resolution by slices with associated markings $\mu_{X_i}= \mu_{0}^i, \mu_{1}^i, ...$, and build a model manifold $M_{H_i}$ from these data.   A submanifold of the model admits a $K(S_i)$-bi-Lipschitz homeomorphism onto a perhaps smaller neighborhood $E_{S_i}'\subset E_{S_i}$.  Now we apply Construction \ref{chains} to obtain a sequence of $3$-chains $\{V_k^i\}_{k=1}^\infty$.  By Lemma \ref{lowerbound},  the function $k\mapsto  |\hat\omega_i(V^i_k)|$ is proper and increasing.  We have established one of the two main estimates needed to prove the corollary in the case that $S_i$ is compressible.  

Now we must argue that for $j\not=i$, the chains straightened with the $j$-metric have volume uniformly bounded from above.  An application of Lemma \ref{uniformupperbound} will complete the proof of the corollary.  
To this end, let ${f_i}_*\inverse (\g_i(S_i)) = \g(S_i)$ and $\tilde{f}_j : M_{\g(S_i)} \to M_{\g_j(S_i)}$ be a lift of $f_j$ to the corresponding covers.  Let $g_i: M_{\g_i(S_i)}\to M_{\g(S_i)}$ be a smooth homotopy inverse for $\tilde{f}_i$.  We need to argue that for $i\not= j$, there is a $C_j>0$ such that $|\hat\omega_j((\tilde{f}_j\circ g_i)(V^i_k))|<C_j$.  We claim that if $\alpha^i_k\subset S_i$ are simple curves satisfying $\alpha^i_k\to \nu(S_i)$ as $k\to \infty$ then the geodesic representatives of $(\tilde{f}_j\circ g_i )(\alpha^i_k)$ in $M_{\g_j(S_i)}$ must stay in a compact set for all $k$. Since $\nu(S_i)$ is not an end invariant of $M_{\g_j(S_i)}$, by Theorem \ref{NS} $\nu(S_i)$ is realized in $M_{\g_j(S_i)}$.  Thus there is a compact set $K_j\subset M_{\g_j(S_i)}$ so that $K_j\supset ({\tilde{f}_j\circ g_i})(\alpha^i_k)^*$ for every $k$, which is exactly what we wanted. By construction, the boundary of our three chains $V^i_k$ define triangulations $\cT^i_0$ and $\cT^i_k$ of $S_i$.  Let $G^{i,j}_k: S_i\to M_{\g_j(S_i)}$ be simplicial hyperbolic surfaces adapted to $\cT^i_k$ homotopic to ${\tilde{f}_j\circ g_i}:S_i\to M_{\g_j(S_i)}$. Applying the Bounded Diameter Lemma to $(G^{i,j}_k,\cT^i_k)$, we have a compact set $K_j'$ such that $\im G_k^{i,j}\subset K_j'$ for all $k$.  Apply Lemma \ref{straighten} and Lemma \ref{uniformupperbound} to obtain an upper bound analogous to (\ref{UPPER}).  Combining this paragraph with the previous we have that any co-boundary $B$ with $\delta B = \hat\omega_i-\hat\omega_j$ must satisfy an inequality analogous to Inequality (\ref{contradiction}), and so $B$ is not a bounded co-boundary.  That is, $[\hat\omega_i]\not= [\hat\omega_j]$ if $i\not=j$.  Linear independence follows as it did in Theorem \ref{mainthm1}.
\end{proof}

\section{Infinite ends are cohomologically separated}\label{main2}

We show that we may build triangulations of the manifolds $W_k$ from \S\ref{below} that are efficient in the sense that the average volume of a straight simplex is bounded below for each $k$.   Brock supplies us with such efficient triangulations in \cite{brock:wp} that we augment with a kind of `wrapping trick' in \S\ref{telescopesoftori}.
Combining his construction with the hierarchy machinery of Masur-Minksy \cite{masur-minsky:complex2} and our work in \S\ref{below}, we bound ${\|[\hat\omega_1]-[\hat\omega_2]\|_\infty}$ from below by a constant that only depends on the topology of $S$. 
Our plan is to relate the size of a hierarchy path to progress in $\mathbf{P}(S)$, the pants graph for $S$.  Curves which appear in a geodesic in a hierarchy path come in two flavors---curves which appear on geodesics in the hierarchy $H$ whose domain is a four-holed sphere or a one-holed torus contribute substantial volume, while others do not.    We collect a number of useful results, below.  

\subsection{Counting in a hierarchy}
Let $\rho: \pi_1(S) \to \G$ be a marked Kleinian surface group with at least one geometrically infinite end and not parabolic cusps.  Let $\nu=(\nu^-,\nu^+)$ be the end invariants for $M_\rho$. Construct a hierarchy $H_{\nu}$, and fix a resolution by slices with corresponding markings $\{\mu_i\}_{i\in J}$ where $J\subset\bZ$ is a (bi)-infinite interval; denote $\base(\mu_i) = \cP_i$.  

Suppose $i_0<i_1$ and consider the markings $\mu_0$ and $\mu_1$ coming from the $i_0$ and $i_1$ slices of $H_\nu$ respectively.  The hierarchy $H_\nu$ naturally restricts to a hierarchy $H$ with $\mathbf{I}(H)=\mu_0$ and $\mathbf{T}(H)=\mu_1$ and so that the resolution of $H_\nu$ by slices $\{\mu_i\}_{i\in J}$ restricts to a resolution with slices $\{\mu_i : {i_0\le i \le i_1}\}$ of $H$ (see \S3.3 or \cite{durham:hierarchies} Appendix 8, Lemma 8.3.1, Lemma 8.3.2).  If $h = v_0, ... ,v_n$, then $|h| = n$ and the \emph{size} of a hierarchy path $H$ is defined by $|H|= \sum_{h\in H} |h|$.   
Recall that if $h\in H$, then the \emph{domain} $D(h)$ of $h$ is the subsurface on which $h$ is supported, and $\xi(D(h)) = 3g+n$ where $g$ is the genus of $D(h)$ and $n$ is the number of its boundary components or punctures.  The \emph{non-annular size} of a hierarchy $H$ is \[\|H\| = \sum_{\substack{h\in H\\ \xi(D(h))\not=2}}|h|\]
\begin{theorem}[\cite{masur-minsky:complex2}, Theorem 6.10]\label{counting}
There are constants $M_2$ and $M_3$ depending only on $S$ such that if $\mu$ and $\mu'$ are complete markings of $S$ and $H$ is a hierarchy with initial and terminal markings $\mu$ and $\mu'$, then \[M_2\inverse d_\mathbf{P}(\base(\mu),\base(\mu'))-M_3\le \|H\| \le M_2d_\mathbf{P}(\base(\mu),\base(\mu')). \]
\end{theorem}
\vspace{.5cm}

Define the \emph{$4$-size} of $H$ by \[\|H\|_4 = \sum_{\substack{h\in H \\ \xi(D(h))=4}} |h|.\]
Minsky shows that that non-annular size of $H$ and the $4$-size of $H$ are comparable; the following is an application of Proposition 9.7 of \cite{minsky:ELTI} and some elementary observations about the subsurfaces that appear as domains in a finite hierarchy that share some collection of boundary components.  

\begin{theorem}[\cite{minsky:ELTI}, Proposition 9.7]\label{four}
There are constants $M_4$ and $M_5$ depending only on $S$ such that a complete finite hierarchy $H$ satisfies \[M_4\inverse \|H\| -M_5 \le \|H\|_4\le \|H\|. \]
\end{theorem}
Let us combine Theorems \ref{counting} and \ref{four} into the following statement.  For a complete finite hierarchy $H$ with $\mathbf{I}(H)=\mu_0$ and $\mathbf{T}(H)=\mu_1$
\begin{equation}\label{useful}
\|H\|_4\ge M_6d_{\mathbf{P}}(\base(\mu_0), \base(\mu_1)) - M_7
\end{equation}
where $M_6 = (M_4M_2)\inverse$ and $M_7 = M_5\inverse M_3+M_5$.  

\subsection{Hierarchies and efficient triangulations}
Given a finite hierarchy $H$, we would like to build a triangulation of $S\times I$ interpolating between $\mathbf{I}(H)$ and $\mathbf{T}(H)$.  First, we will construct a class of triangulations of $S$ that incorporate the twisting information encoded in the initial and terminal markings of $H$; our triangulations $\cT$ adapted to a pants decomposition $\cP$ from earlier sections did not incorporate such data.  Our new triangulations are stable under certain moves, and any resolution of $H$ will yield a collection of such moves that allow us to construct a triangulation of $S\times I$ with the stipulated boundary data.  Moreover, each step forward in the resolution of $H$ will contribute a uniformly bounded number of tetrahedra to the final triangulation.  We will therefore have an upper bound for the 1-norm of our triangulation in terms of $|H|$.  The construction of these triangulations and their properties are due to Brock in \S5 of \cite{brock:wp}.  

Let $Y$ be a pair of pants.  A \emph{standard} triangulation $\cT_Y$ of $Y$ has the following properties:
\begin{enumerate}
\item $\cT_Y$ has two vertices on each boundary component.
\item $\cT_Y$ has two disjoint spanning triangles with no vertices in common, and a vertex on each component of $\partial Y$.
\item The remaining 3 quadrilaterals are diagonally subdivided by an arc that travels ``left to right'' with respect to the inward pointing normal to $\partial Y$ (See Figure \ref{standardpants}).
\end{enumerate}

We construct a \emph{standard triangulation} $\cT$ suited to a pants decomposition $\cP$ by gluing together standard triangulations on pairs of pants as in (1) - (3) below. Note that by an Euler characteristic computation, $|\cT| = 16(g-1)$.
\begin{figure}[h]
\includegraphics[width=3in]{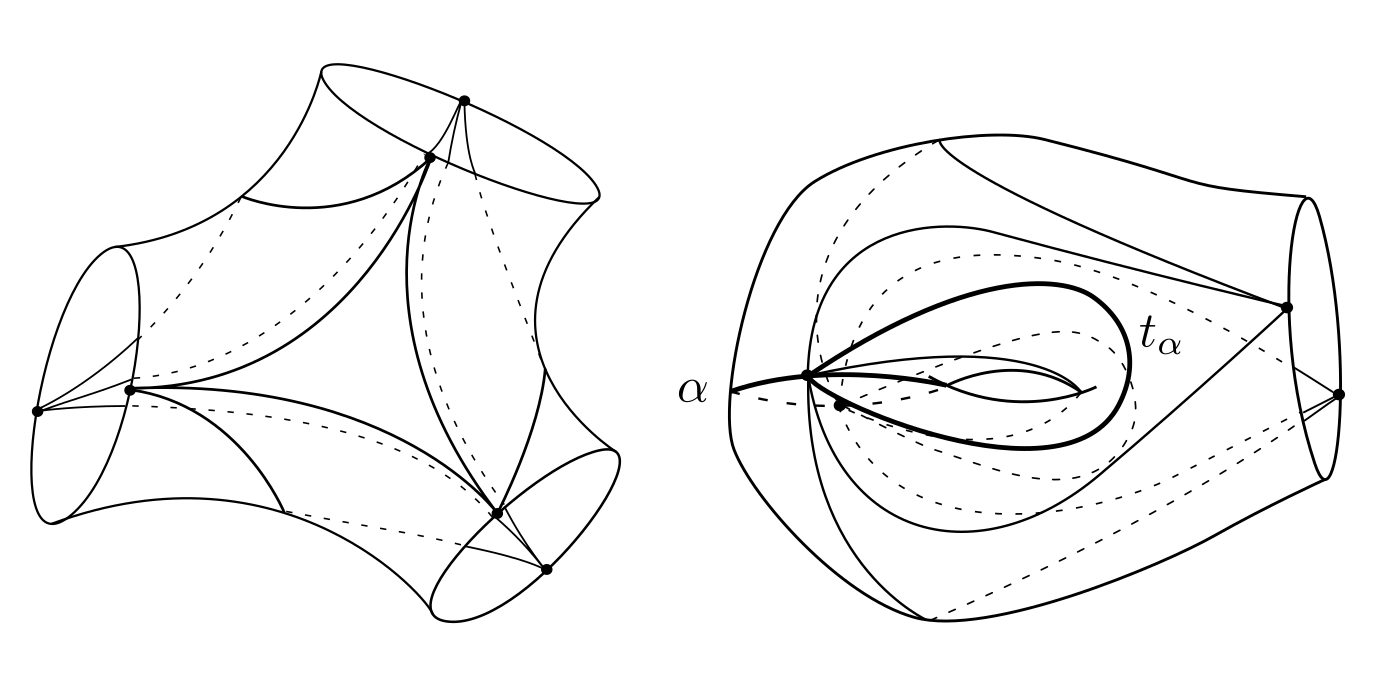}
\caption{On the left is a standard triangulation $\cT_{\cP}$ of a pair of pants $Y$.  On the right is a surface $S_\alpha$ in the case that it is a $1$-holed torus.  The canonical transversal $t_\alpha$ is in bold.  }
\label{standardpants}
\end{figure}
\begin{enumerate}
\item $\cT$ has two vertices $p_\alpha$ and $\bar p_\alpha$ on each component $\alpha$ of $\cP$, and two edges $e_\alpha$ and $\bar e_\alpha$ in the complement of $\alpha\setminus \{p_\alpha, \bar{p}_\alpha\}$.
\item If $Y$ is a complementary open pair of pants in $S\setminus \cP$, the restriction of $\cT$ to $Y$  is a standard triangulation of $Y\cup\partial Y$.
\item If two boundary components $\alpha_1$ and $\alpha_2 \subset \partial Y$ represent the same curve in  $S$, then the edge of each spanning triangle that runs from $\alpha_1$ to $\alpha_2$ forms a closed loop (see Figure \ref{standardpants}).
\end{enumerate}

\begin{figure}[h]
\includegraphics[width=5in]{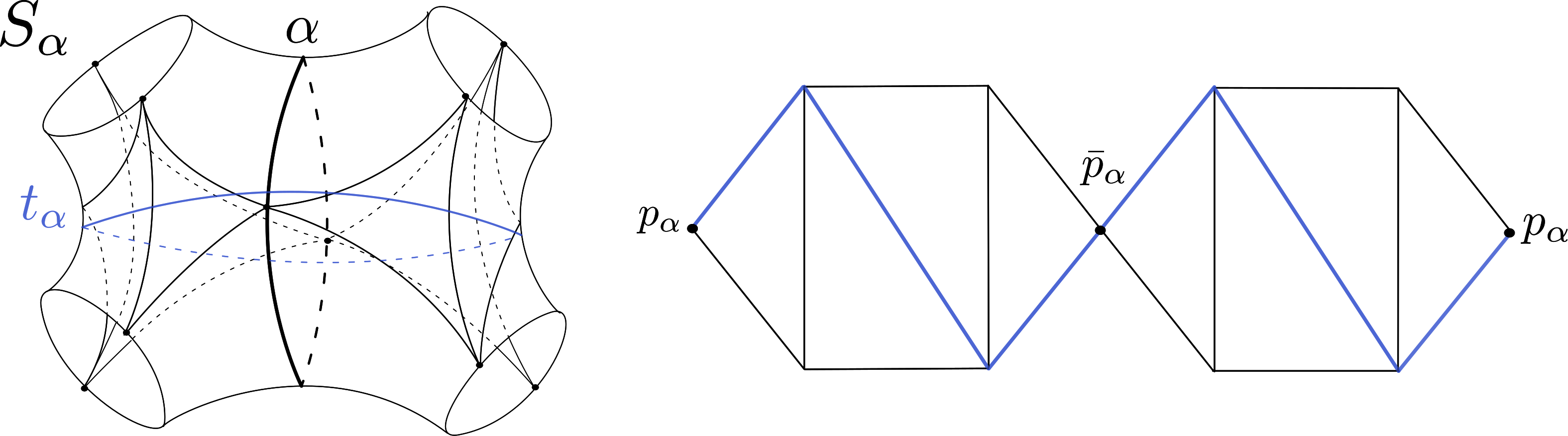}
\caption{Here $S_\alpha$ has genus $0$.  On the left, only the edges of $\cT$ restricted to $S_\alpha$ that contribute to the construction of $t_\alpha$ are drawn, and the homotopy class of $t_\alpha$ is in blue.  On the right, we have the two hexagons $X_l$ and $X_r$ and the concatenation of the blue segments forms the homotopy class of $t_\alpha$. }
\label{transversalcurves}
\end{figure}

We want to outfit our triangulation to include the twisting information of the transversals in the initial and terminal markings of $H$.  Suppose $\mu$ is a complete marking on $S$ with $\base(\mu)=\cP$.  An important property of a standard triangulation $\cP$ suited to $\cT$ is that it defines canonical homotopy classes of transversal curves $t_\alpha \in \cC(S)$ for each $\alpha \in \cP$.  Once we have identified the canonical transversal curve, we can apply Dehn twists $T_\alpha^n(\cT)$ about each pants curve to obtain a triangulation $\cT'$ suited to $\cP$ and outfitted to $\mu$.  The standard triangulation $\cT'$ is \emph{outfitted} to $\mu$ if the projection $\pi_{\alpha}(t_\alpha)$ of the canonical transversals $t_\alpha$ defined by $\cT'$ onto the annular curve graph for $\alpha$  coincides with the transversal $t$ for the pair $(\alpha,t)\in \mu$. We now explain how to identify the homotopy class of the canonical transversal.  Let $S_\alpha$ be the component of $S\setminus (\cP\setminus\{\alpha\})$ containing $\alpha$.  It may be helpful to refer to Figure \ref{standardpants} for $(a)$ and Figure \ref{transversalcurves} for $(b)$.
\begin{enumerate}[(a)]
\item If $S_\alpha$ is a $1$-holed torus, then each spanning triangle for $\cT$ in $S_\alpha$ has one edge with its endpoints on $\alpha$ identified.  These edges are in the same homotopy class, call this homotopy class $t_\alpha$.  
\item If $S_\alpha$ is a $4$-holed sphere, let $Y_r$ and $Y_l$ be the two components of $S_\alpha\setminus \alpha$.  The restriction of $\cT$ to $Y_r$ contains two spanning triangles and one annulus in their complement disjoint from $\alpha$; call the union of these triangles and annulus $X_r$.  The boundary of $X_r$ is a hexagon; we choose a path in the $1$-skeleton of $X_r$ called $t_{\alpha,r}$ that does not contain an arc of $\partial S_\alpha$ and that joins $p_\alpha$ with $\bar{p}_\alpha$.  Construct $t_{\alpha,l}\subset Y_l$ analogously.  The concatenation of the edges of  $t_{\alpha,l}$ and $t_{\alpha,r}$ forms a homotopy class of curves $t_\alpha$, and any other choice of paths in $X_r$ and $X_l$ are in the same homotopy class.
\end{enumerate}

Notice that the geometric intersection number $i(t_\alpha,\alpha) = 1$ if $S_\alpha$ has genus $1$, and $i(t_\alpha, \alpha)=2$ otherwise.  Furthermore, since $t_\alpha\subset S_\alpha$ and $\{t_\alpha, \alpha\}$ fills $S_\alpha$, the collection $\mu' = \{(\alpha, t_\alpha): \alpha\in \cP\}$ defines a clean marking compatible with $\mu$.

Brock considers three types of \emph{moves} on standard triangulations corresponding to elementary moves on clean complete markings.  The Dehn twist move on standard triangulations (\cite{brock:wp}, \textbf{MVI}) corresponds to a twist move on complete clean markings, while the genus $1$ and genus $0$ moves on standard triangulations (\cite{brock:wp}, \textbf{MVII} and \textbf{MVIII} respectively) both correspond to the \emph{flip} elementary moves on clean complete markings, depending on whether the subsurface of complexity $4$ on which the move takes place has genus $1$ or $0$.   For each type of move, there is a \emph{block triangulation} realizing this move in a triangulated model manifold $S\times I$.  Suppose inductively that $\cT$ is a triangulation of $S\times I$ such that the boundary of $\cT$ is triangulated by standard triangulations suited to complete clean markings $\mu$ and $\mu'$.  The block triangulation associated to each move is a recipe for attaching tetrahedra to $\cT$ to obtain a new triangulation $\cT'$ with boundary data stipulated by markings $\mu$ and $\mu''$, where $\mu''$ and $\mu'$ are related by an elementary move.  Every elementary move on complete clean markings is realized by a move on standard triangulations, and every move on standard triangulations is realized by a uniformly bounded number of block triangulations; finally, each block triangulation uses a uniformly bounded number of tetrahedra.   

There may be a large number (unbounded amount) of Dehn twist moves needed to advance from triangulations adapted to pants decompositions $\cP$ to $\cP'$ even if they are adjacent in $\mathbf{P}(S)$. We discuss the Dehn twist moves and their block triangulations in \S\ref{telescopesoftori} in more detail.  In particular, instead of using one block for every Dehn twist move needed to advance from markings $\mu$ and $\mu'$ compatible with adjacent pants decompositions $\cP$ to $\cP'$, we will use {\it just one block}.  The trade off is that our modified block which does the work of $n$ Dehn twist blocks will no longer  \emph{triangulate} the model manifold; it will however be represented by a real $3$-chain with uniformly bounded $1$-norm.

From the proof of Theorem 5.7 in \cite{brock:wp}, we extract the following

\begin{lemma}[\cite{brock:wp}, Theorem 5.7]\label{goodtri}
Let $S$ be a closed surface and let $\mu_0$ and $\mu_1$ be complete markings on $S$.  Let $H$ be a hierarchy of tight geodesics with $\mathbf{I}(H)=\mu_0$ and $\mathbf{T}(H)=\mu_1$, and let $\cT_i$ be a standard triangulation of $S$ suited to $\mu_i$, for $i=0,1$.  Then there is a triangulation $U$ of $S\times I$ and a universal constant $M_8>0$ with the following properties: 
\begin{enumerate}[(i)]
\item $\partial U = \cT_1\times \{1\}-\cT_0\times\{0\}$,
\item $U$ decomposes as a sum $U = U_{DT}+ U_0$, where $U_{DT}$ is the collection of tetrahedra which arise from Dehn twist blocks.
\item $|U_0|\le M_8\|H\|$.
\end{enumerate}

\end{lemma}

\subsection{Telescopes of tori}\label{telescopesoftori}

Let $A = S^1 \times [0,1]$ be an annulus with core curve $\alpha = S^1 \times \{1/2\}$.  Consider the triangulation $\cT_A$ pictured in Figure \ref{tori}; it determines an integral $2$-chain $T_A\in \Cb_2(A;\bR)$.  Let $D_\alpha: A\to A$ be a right Dehn twist about $\alpha$.  Then $(D_\alpha)_*T_A$ is also an integral $2$-chain representing a triangulation $(D_\alpha)_*\cT_A$ of $A$.  

\begin{figure}[h]
\includegraphics[width = 5in]{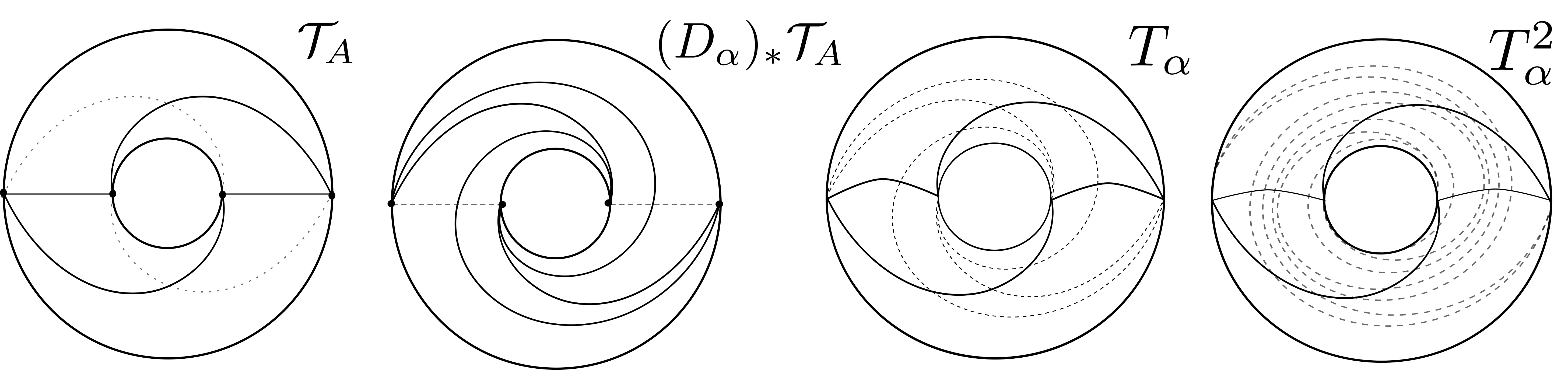}
\caption{The dotted edges in the annuli indicate the diagonal switches for $\cT_A$ and $(D_\alpha)_*\cT_A$ that produce a common triangulation of the annulus.  This results in the triangulation $\cS_\alpha$ of the solid torus $\bT_\alpha$ with boundary represented by the chain $T_\alpha$.  } 
\label{tori}
\end{figure}

Now consider two copies of $A$ called $A_L$ and $A_R$.  Glue these annuli via the identity mapping of their boundaries to obtain a torus $A_L\cup_\partial A_R$, and let $\iota_\bullet : A_\bullet \to A_L\cup_\partial A_R$ be the restriction of the quotient mapping for $\bullet = R,~L$.  Then $T_\alpha : = (\iota_L)_*T_A - (\iota_R\circ D_\alpha)_*T_A$ is a fundamental cycle for the torus, and it represents a triangulation $\cT_\alpha$ of $A_L \cup_\partial A_R$.  Moreover, since $\cT_A$ and $(D_\alpha)_*\cT_A$ differ by four diagonal switches that form a common refinement of $A$ (see Figure \ref{tori} as well as \cite{brock:wp} \S5.2), $\cT_\alpha$ determines a triangulation $\cS_\alpha$ on a solid torus $\bT_\alpha$ with core curve $\alpha$ obtained by filling in the 4 tetrahedra which interpolate between $\cT_A$ and $(D_\alpha)_*\cT_A$.  Finally, $\cS_\alpha$ determines an integral $3$-chain that we call $S_\alpha\in \Cb_3(\bT_\alpha;\bR)$.  Observe that $\partial S_\alpha = T_\alpha$ and $\|S_\alpha\|_1 = 4$.

More generally, we \emph{stack} $n+1$ copies of $A$, all glued by the identity mapping on their boundaries to form the complex $A_0 \cup_\partial A_1\cup_\partial ... \cup_\partial A_n$.  We also have restrictions of the quotient map $\iota_\bullet: A_\bullet \to A_0 \cup_\partial ... \cup_\partial A_n$, for $\bullet = 0, ..., n$.
We can find an integral chain $S_\alpha^n\in \Cb_3(\bT_\alpha;\bR)$ such that
 \begin{equation}\label{property1}
  \partial S_\alpha^n = (\iota_0)_*T_A - (\iota_n\circ D_\alpha^n)_* T_A=: T_\alpha^n,
  \end{equation}
 and such that 
 \begin{equation}\label{property2}
 \|S_\alpha^n\|_1 = 4n.
 \end{equation}
 This is because $(D_\alpha^{k-1})_*\cT_A$ and $(D_\alpha^{k})_*\cT_A$ also differ by four diagonal switches, so they determine a triangulation of a solid torus bound by $A_{k-1}\cup_\partial A_k$ inducing a $3$-chain that we call $(D_\alpha^k)_*S_\alpha$ such that $\partial (D_\alpha^k)_*S_\alpha = (\iota_{k-1}\circ D_\alpha^{k-1})_*T_A - (\iota_{k}\circ D^k_\alpha)_*T_A$.    Define 
 \[S_\alpha^n : = \sum_{k = 0}^{n-1}(D_\alpha^k)_* S_\alpha,\]
  and check that it satisfies the properties (\ref{property1}) and (\ref{property2}) stated above.  Note that what we have done here is no different from stacking \emph{right Dehn twist blocks} as in \cite{brock:wp}, but it will be convenient to have a name for the chain $S_\alpha^n$.  
  
Later in this section, we will be unhappy with the fact that $\|S_\alpha^n\|_1 = 4n$, so we will build a real $3$-chain $\sS_\alpha^n$ such that $\partial \sS_\alpha^n =T_\alpha^n$ and $\| \sS_\alpha^n\|_1<m_0$; we will see that we can actually take $m_0 = 37$, but the point is that it is uniformly bounded, not depending on $n$.  We will use the fact that covers of solid tori are themselves solid tori.  So, let $\pi : \widetilde{\bT_\alpha} \to \bT_\alpha$ be a two to one covering, and find a homeomorphism $s: \bT_\alpha \to \widetilde{\bT_\alpha} $ such that $(\pi\circ \partial s)_*\alpha = 2\alpha \in \pi_1(\partial \bT_\alpha)$.   

For any $n$, we claim that there is a $c^n \in \Cb_3(\bT_\alpha;\bR)$ such that 
\begin{equation}\label{property3}
\partial c^n = T_\alpha^n - \frac{1}{2}(\pi\circ \partial s)_*T_\alpha^n,
\end{equation}
and such that
\begin{equation}\label{property4}
\|c^n\|_1 = 18.
\end{equation}
The construction of $c^n$ is most easily seen in the universal cover $ \bR^2\times [0,\infty) $ of $\partial \bT_\alpha \times [0, \infty)$.  Here, we are thinking of $\partial \bT_\alpha \times [0, \infty)$ as $\bT_\alpha$ away from its core curve. We remark that a very similar construction can be found in \S4.4 of  \cite{sisto:zero}, and we are grateful to Maria Beatrice Pozzetti for the observation that it may be helpful for us, here.  

Identify $\pi_1(\partial \bT_\alpha)$ with the $\bZ$-span of $\{[\alpha], [m]\}$, where $m$ is the meridian of $\bT_\alpha$. We have an action of $\pi_1(\partial \bT_\alpha)$ on $\bR^2\times [0,\infty)$ by the formula $[\alpha]. (x,y, t) = (x+1, y , t)$ and $[m].(x,y,t) = (x, y+1, t)$. With this identification, we have the orbit projection map $p : \bR^2\times [0,\infty) \to  \partial \bT_\alpha \times [0, \infty)$ as well as a retraction $r: \partial \bT_\alpha\times [0,\infty) \to \partial \bT_\alpha\times \{0\}$.  

\begin{figure}[h]
\includegraphics[width=3in]{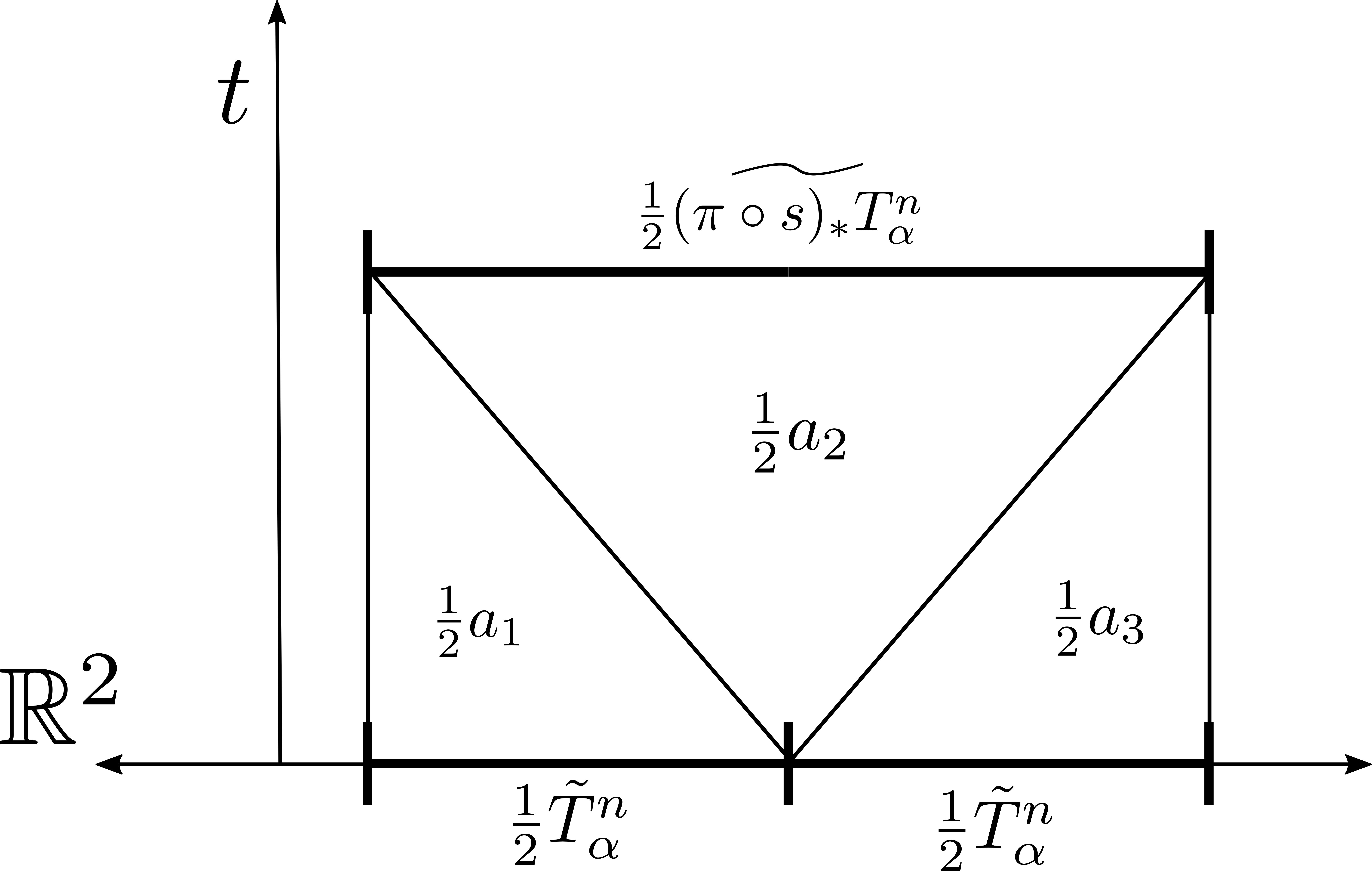}
\caption{Each of the $a_i$ are triangulations of the prisms joining the connected lift $\tilde{T}_\alpha^n$ with the left or right most edges of the connected lift $\widetilde{(\pi\circ s)_* T_\alpha^n}$. Define $\frac{1}{2} c^n_0:= \frac{1}{2}(a_1+a_2+a_3)$.}
\label{telescopespic}
\end{figure}

To construct the chain $c^n$, first place two  connected lifts of $T^n_\alpha$ to $\bR^2\times [0,\infty)$ next to each other on the plane $t=0$ and one lift of $(\pi\circ \partial s)_*T_\alpha^n$ on the plane $t=1$.  Using straight line homotopies, we  join these two levels in the pattern pictured in Figure \ref{telescopespic}.  The lifted triangulations determine cell structures on the images of the straight line homotopies, which we can triangulate using $12$ tetrahedra each for a total of 36 tetrahedra.  Call the chain induced by this triangulation $c^n_0$.  Then $c^n:=\frac{1}{2} (r\circ p)_* c^n_0$ satisfies Properties (\ref{property3}) and (\ref{property4}).

Choose $K_n\ge \lfloor \log_2 4n\rfloor +1$, and define 

\[ d_n = \sum_{k = 0}^{K_n-1} 2^{-k}(\pi\circ s)_*^k c^n.\]
 
 One checks that $\partial d_n = T_\alpha^n -2^{-K_n} (\pi\circ s)_*^{K_n}T^n_\alpha$.  Define \[\sS_\alpha^n : = d_n + 2^{-K_n} (\pi\circ s)_*^{K_n}S_\alpha^n.\]  The point of this section was to construct the chain $\sS_\alpha^n$; we have
 
 \begin{lemma}\label{telescope}	
 The chain $\sS_\alpha^n \in \Cb_3(\bT_\alpha;\bR) $ satisfies the properties \begin{enumerate}[(I)] 
 \item $\partial \sS_\alpha^n  = T_\alpha^n$ and
 \item $\|\sS_\alpha^n\|_1\le 37$.
 \end{enumerate}
 \end{lemma} 
 \begin{proof}
Property (I) follows directly from the construction.  For Property (II), we have
\[\begin{aligned}
\|\sS_\alpha^n\|_1 &\le  \|d_n\|_1  +2^{-K_n} \|(\pi\circ s)_*^{K_n}S^n_\alpha\|_1 \\
& \le \sum_{k = 0}^{K_n-1} 2^{-k}\|(\pi\circ s)_*^k c^n\|_1 + \frac{\|S_\alpha^n\|_1}{2^{K_n}}\\
 &\le\sum_{k = 0}^{\infty} 2^{-k}\|c^n\|_1 + \frac{4n}{2^{K_n}} \le 18\cdot 2 +1=37.
\end{aligned}\]

\end{proof}

\subsection{The modified Dehn twist block}
In this section, we will modify the $n$ adjacent Dehn twist blocks associated to a single curve $\alpha$ that arise in the construction of the model manifold from \cite{brock:wp} Theorem 5.7.  These blocks are triangulated and so determine an integral $3$-chain $B_\alpha \in \Cb_3(S\times I ;\bR)$.  In fact, the chain $U_{DT}$ from Lemma \ref{goodtri} decomposes further as 
\[U_{DT} = \sum_{\substack{h_\alpha\in H \\ \xi(D(h_\alpha))=2}} B_\alpha.\]
Since $\|B_\alpha\|_1$ can be arbitrarily large (it's roughly proportional to both $|h_\alpha|$ and $d_{Y_\alpha} (\nu^-, \nu^+ )$), our goal in this section is to replace each $B_\alpha$ with a real $3$-chain $\mathscr{B}_\alpha \in \Cb_3(S\times I; \bR)$ such that $\partial B_\alpha = \partial \mathscr{B}_\alpha$ and $\|\mathscr{B}_\alpha \|_1 <m_0$.  

Given a pants decomposition $\cP$ and a curve $\alpha\in \cP$,  the \emph{subsurface block} $B_\alpha$ is the quotient \[B_\alpha = S_\alpha \times [0,1] / (x,t) \sim (x,0) \text{ for } x\in \partial S_\alpha, ~ t\in [0,1]. \]
Depending on the genus of the surface $S_\alpha\subset S$ (either $0$ or $1$), Brock describes a \emph{standard block triangulation} $\cT_\alpha$ of $B_\alpha$ which extends the cell structure on $B_\alpha$ given by $\cT_P|_{S_\alpha}\times [0,1]/\sim$. The precise details of this construction are not relevant for us, except that the number of tetrahedra in $\cT_\alpha$ is at most some $n_0$, no matter the genus of $S_\alpha$.

We must revisit the construction of the Dehn twist block from \cite{brock:wp}, \S5.2.  The difference is precisely the following:  In the description of the Dehn twist block (\textbf{BLI} in \cite{brock:wp}), we must start with the standard block triangulation $\cT_\alpha$ of $B_\alpha$.  Consider the annulus $A= \alpha \times [0,1]$ in $B_\alpha$ with the triangulation $T_A$ on $A$ induced by $\cT_\alpha$.  Cut $B_\alpha$ along $A$ to obtain two annuli $A_L$ and $A_R$ that bound the \emph{local} left and right side of $B_\alpha \setminus A$.  Re-glue the $\alpha\times\{0\}$ boundary components of $A_L$ and $A_R$ by the identity, and re-glue the $\alpha\times\{1\}$ boundary components of $A_L$ and $A_R$ shifted by $n$ Dehn twists.  Via this surgery operation, we have a triangulation of the closure of $B_\alpha\setminus A$ which defines an integral $3$-chain with  boundary $ T_\alpha^n$.  Add to this the  $3$-chain $\sS_\alpha^n$ to obtain $\mathscr{B}_\alpha$.  By Lemma \ref{telescope}, $\mathscr{B}_\alpha$ satisfies \begin{enumerate}[(I)]
\item $\partial \mathscr{B}_\alpha = \partial B_\alpha$
\item $\| \mathscr{B}_\alpha\|_1 \le 37 +n_0$.
\end{enumerate}
Define

 \[\mathscr{U}_{DT} := \sum_{\substack{h_\alpha \in H \\ \xi(D(h_\alpha))=2}} \mathscr{B}_\alpha.\]

With notation as in Lemma \ref{goodtri}, recall that $U = U_{DT} + U_0$, and define \[\mathscr{U} : = U_0 +\mathscr{U}_{DT}.\]

The following proposition is then a direct consequence of Lemma \ref{goodtri} and (I) and (II) above.

\begin{prop}\label{reallygoodtri} The chain $\mathscr{U}\in \Cb_3(S\times I; \bR)$ satisfies the following two properties
\begin{enumerate}[(i)]
\item $\partial \mathscr{U} = \cT_1\times \{1\}-\cT_0\times\{0\}$
\item $\|\mathscr{U}\|_1 \le M_8' \|H\|,$ 
\end{enumerate}
and  $M_8' = M_8(37+n_0)$.  

\end{prop}

\subsection{Lower bounds on norm}

We use the collection of facts from above and our results from \S\ref{below} to give a more refined lower bound for $|\hat\omega(V)|$.  

\begin{prop}\label{finelower}
Let $\mu_{i_0}$ and $\mu_{i_1}$ be complete markings associated to slices of a resolution of $H_\nu$ such that $i_0<i_1$ and $\base(\mu_{i_0}) \cap \base(\mu_{i_1})=\emptyset$. Suppose $F_i: S\to M_\g$ are simplicial hyperbolic surfaces whose associated triangulations $\cT_i$ are outfitted to $\mu_{i_k}$ for $k=0,1$, and suppose that $H$ is the subhierarchy path of $H_\nu$ with  $\mathbf{I}(H) = \mu_{i_0}$ and $\mathbf{T}(H) = \mu_{i_1}$.  There are constants $M_9$ and $M_{10}$ depending only on $S$ and a $3$-chain $V$ such that
\begin{enumerate}[(i)]
\item $\partial V = F_1\cT_1-F_0\cT_0$;
\item $\|V\|_1\le M_8'\|H\|$;
\item $|\hat\omega(V)| \ge M_9d_{\mathbf{P}}(\cP_0,\cP_1)-M_{10}$.
\end{enumerate}
 
\end{prop}
\begin{proof} First we define $V$.  Let $H:S\times I \to M_\rho$ be any homotopy between $F_0$ and $F_1$ and set $H_*(\mathscr{U})=V$, where $\mathscr{U}$ is the chain from Proposition \ref{reallygoodtri} built from the data in our hypotheses.  Notice that $\partial\str(V) = \str\partial V = \str(F_1\cT_1 - F_0\cT_0) = F_1\cT_1 - F_0\cT_0$.  Properties (i) and (ii) are established.  

The following is essentially just a refinement of the work done in Lemma \ref{lowerbound}.  We have done some additional bookkeeping and have a more refined estimate for the volume of the submanifold $W$, which we now include.  So, proceeding as in the proof of Lemma \ref{lowerbound}, find extended split level surfaces $G_0$ and $G_1$ with disjoint images bounding a submanifold $W\subset M_\rho$ homeomorphic to $S\times I$ as in Proposition \ref{Split}.  The submanifold $W$ is the $K$-bi-lipschitz image of a union of \emph{blocks} and \emph{tubes} in the model manifold $M_\nu$.  There are two isometry types of blocks, and there is one block for every edge in a 4-geodesic, \ie in a geodesic $h\in H$ such that $D(h)=4$ (see \cite{minsky:ELTI}, \S8).  The volume of each block in the model metric is at least a constant $M_9'$.  Thus, \[\vol(W)\ge M_9'K^{-3}\|H\|_4.\]
By Inequality (\ref{useful}), this implies in particular that 
\[\vol(W)\ge M_6M_9'K^{-3}d_\mathbf{P}(\cP_{i_0},\cP_{i_1})-M_7.\]  Define $M_9 = M_6M_9'K^{-3}$. We now have the more refined estimate for the volume of $W$ that we were seeking.  We will now finish the proof and establish (iii), just as we did in Lemma \ref{lowerbound}.  

Build homotopies as in Lemma \ref{lowerbound} from $G_0$ to $F_0$ and $G_1$ to $F_1$ triangulated by $C_0$ and $C_1$.  Any triangulation of $W$ (also called $W$) gives us a $3$-chain $V' = C_1+W+C_0$, and \[\partial V' = F_1\cT_1 - F_0\cT_0=\partial V.\]
Thus, by a similar computation as in the proof of Lemma \ref{lowerbound},  \[|\hat\omega(V)| = |\hat\omega(V')|\ge \vol(W)-2M_0\ge M_9d_\mathbf{P}(\cP_{i_0},\cP_{i_1})-M_7-2M_0,\]  where $M_0$ is the constant from Lemma \ref{hmtpy}.  Set $M_{10} = M_7-2M_0$.  This establishes Property (iii).  
 \end{proof}
 
 We are ready to state the main result of this section. The idea here is to show that the volume of the `average simplex' in our sequence of $3$-chains must be bounded below.  We will see that this lower bound on the average simplex supplies a lower bound on the norm of $\|\hat\omega_\rho-\hat\omega_{\rho'}+\delta B\|_\infty$ for any bounded $2$-co-chain $B$, when $\rho$ and $\rho'$ have different geometrically infinite end invariants.  We use our more refined lower bounds for volume and upper bounds for the number of tetrahedra required to triangulate the homotopy bounded by two simplicial hyperbolic surfaces to obtain the desired bound on the average simplex.
 
 Here we show that the bounded fundamental classes for geometrically infinite surface groups are \emph{uniformly separated} in pseudo-norm, by a constant that depends only on the genus of $S$.  This is a fairly strong strengthening of \cite{soma:boundedsurfaces}, Theorem A.  We remark that there is no reason to believe that our $\epsilon$ in the statement of Theorem \ref{separation} is in any way optimal.  It would be interesting to know what the optimal constant is.  We know, for example, that it must be in the interval $(0, 2v_3]$ (see \cite{soma:boundedsurfaces}, Proposition 3.3).

\begin{theorem}\label{separation}
Let $S$ be a closed orientable surface of genus at least $2$.  There is a constant $\epsilon = \epsilon(S)>0$ such that the following holds. If $\{\rho_j: \pi_1(S) \to \G: j = 1, ... ,n\}$ is a collection of discrete and faithful representations without parabolic elements such that at least one of the geometrically infinite end invariants of $M_{\rho_i}$ is different from the geometrically infinite end invariants of $M_{\rho_j}$ for all $i\not= j$ then 
\[\|\sum_{j = 1}^n a_j[\hat\omega_{\rho_j}]\|_\infty > \epsilon\max{|a_j|}.\]
 \end{theorem}
\begin{proof}
Without loss of generality, assume that $|a_1| = \max{|a_j|}$ where $\rho_1$ has end invariants $\nu = (\nu^-,\nu^+)$, a hierarchy $H_\nu$, and resolution by slices with corresponding markings $\{\mu_i\}$.  Without loss of generality, assume $\nu^+=\lambda \in \cEL(S)$ is the end invariant of $\rho_1$ that is different from all the end invariants of $\rho_j$,  $j>1$.  We construct a sequence of chains $V_k$ analogous to those obtained from Construction \ref{chains} as follows.  For each $i\ge0$, we have a complete marking $\mu_i$.  Let $H_i$ be the sub-hierarchy of $H_\nu$ with $\mathbf{I}(H_i)=\mu_0$ and $\mathbf{T}(H_i)=\mu_i$.   Choose $i$ large enough so that $\cP_{0}\cap \cP_{i}=\emptyset$.  Now for each $i$, apply Proposition \ref{finelower} to obtain a $3$-chain $V_i\in C_3(S\times \bR)$ satisfying properties (i), (ii), and (iii).  
 
Using Theorem \ref{counting} 
we have \begin{equation}\label{above}d_\mathbf{P}(\cP_0,\cP_i)> \frac{\|H_i\|}{M_2}.\end{equation}

By Property (iii) of our sequence $\{V_i\}$ and Inequality (\ref{above}),  \begin{align*} |\hat\omega_{\rho_1}(V_i)|&\ge M_9d_\mathbf{P}(\cP_{0},\cP_{i})-M_{10}\\ &\ge \frac{M_9\|H_i\|}{M_2}-M_{10} \\ &= \epsilon'\|H_i\|-M_{10},
\end{align*}
where $\epsilon'=\frac{M_9}{M_2}$.  

Let $B\in \Cb_b^2(S\times \bR;\bR)$ be arbitrary.  Standard triangulations satisfy $|\cT_i| = 16(g-1)$ for all $i$, so  \[|\delta B(V_i)|=|B(F_{i}\cT_{i} -F_{0}\cT_{0})|\le \|B\|_\infty 32(g-1).  \]  

As in the proof of Theorem \ref{mainthm1}, since $M_{\rho_j}$ does not have $\lambda$ as an end invariant, there is some $c_j>0$ such that 
\[|\hat\omega_{\rho_j}(V_i)|\le c_j\] for all $i$ and $j>1$.   We have 
\begin{align*}
 |(a_1 \hat\omega_{\rho_1}+\sum_{j=2}^n a_j\hat\omega_{\rho_j}+\delta B)(V_i)|&\ge 
|a_1|(\epsilon'\|H_i\| -M_{10})-\sum_{j=2}^n |a_j| c_j- 32(g-1)\|B\|_\infty\\
&=|a_1|\epsilon' \|H_i\|-M_{11}
\end{align*}
On the other hand, using Property (ii) of the sequence $\{V_i\}$, 
\begin{align*}
|(\sum_{j = 1}^n a_j\hat\omega_{\rho_j}+\delta B)(V_i)|&\le \|\sum_{j = 1}^n a_j\hat\omega_{\rho_j}+\delta B\|_\infty\|V_i\|_1\\ &\le \|\sum_{j = 1}^n a_j\hat\omega_{\rho_j}+\delta B\|_\infty M_8'\|H_i\|.
\end{align*}
Thus we have a bound \[\|\sum_{j = 1}^n a_j\hat\omega_{\rho_j}+\delta B\|_\infty\ge \frac{|a_1|\epsilon' \|H_i\|-M_{11}}{M_8'\|H_i\|}=|a_1|\epsilon -\frac{M_{11}}{M_8'\|H_i\|}. \]

Since $\frac{M_{11}}{M_8'\|H_i\|}\to 0$ as $i\to \infty$, and since $B$ was arbitrary, we have shown that \[\|\sum_{j = 1}^n a_j[\hat\omega_{\rho_j}]\|_\infty>|a_1|\epsilon,\] where $\epsilon = \frac{M_9}{M_2M_8'}$ depends only on $S$.  The theorem follows.
\end{proof}
The proof of Theorem \ref{separation} goes through for free groups and fundamental groups of compression bodies, just as in Corollary \ref{general}.  By lifting to the covers corresponding to the ends, we recover Theorem \ref{separation:1} from the introduction.
Corollary \ref{nice map} is now immediate from Theorem \ref{separation} and Corollary \ref{singleboundary}.

\subsection{A criterion for injectivity}

We need an alternate description of the bounded fundamental class of a representation that is more group theoretic and extends to representations that are not necessarily discrete.  Let $\g$ be a countable, discrete group and $\rho: \g\to \G$ be any representation.  Fix a point $x\in \bH^3$, and define $\hat\omega_\rho \in \Cb_b^3(\g;\bR)$ by setting $\hat\omega_\rho (\gamma_0, ..., \gamma_3)$ to be the signed hyperbolic  volume of the convex hull of the $4$-tuple $(\rho(\gamma_0).x, ..., \rho(\gamma_3).x)$.  One checks that $\delta \hat\omega_\rho = 0 $ and $\|\hat\omega_\rho\|\le v_3$.  It is a standard fact that $[\hat\omega_\rho]\in \Hb_b^3(\g;\bR)$ is independent of the choice of $x$, and coincides with the ordinary bounded fundamental class when $\rho$ is a discrete and faithful representation.  Note that if $\gamma\in \ker\rho$, then $\hat\omega_\rho (\gamma^k\gamma_0, \gamma_1, \gamma_2, \gamma_3) = \hat\omega_\rho (\gamma_0, \gamma_1, \gamma_2, \gamma_3)$, for example.

We can now establish a criterion for injectivity. 
\begin{theorem}\label{injective theorem}
Let $S$ be an orientable surface with negative Euler characteristic.  There is a constant $\epsilon' = \epsilon'(S)$ such that the following holds.  Let $\rho: \pi_1(S)\to \G$ be discrete and faithful, without parabolic elements, and at least one geometrically infinite end invariant $\lambda$.  If $\rho': \pi_1(S)\to \G$ is any other representation satisfying \[\|[\hat\omega_\rho] - [\hat\omega_{\rho'}]\|_\infty <\epsilon',\]
then $\rho'$ is faithful.  
\end{theorem}

\begin{proof}
Assume first that $S$ is closed and take $\epsilon'<\epsilon(S)/2$, where $\epsilon$ is the constant from Theorem \ref{separation}.   For sake of contradiction, let $\gamma \in \ker\rho'$ be non-trivial.  There is a finite index subgroup $G\le \pi_1(S)$ such that $\gamma$ lifts to a simple curve in the cover $ i: S'\to S$ corresponding to $G$.  Then $i$ induces a norm non-increasing map on bounded cohomology so that $\|[i^*\hat\omega_\rho] - [i^*\hat\omega_{\rho'}]\|_\infty < \epsilon/2$.

 Let $D_\gamma: S' \to S'$ be a Dehn twist about $\gamma$.  By considering the action of $D_{\gamma*}$ on $\pi_1(S')$, and since $\gamma \in \ker\rho'\circ i$, we see that $D_{\gamma}^*i^*\hat\omega_{\rho'} = i^*\hat\omega_{\rho'}$ at the level of chains.  The mapping class group of $S'$ acts by isometries on $\overline{\Hb}^3_b(\pi_1(S');\bR)$, so we have 
\[\|[D_{\gamma}^*i^*\hat\omega_{\rho}] - [D_{\gamma}^*i^*\hat\omega_{\rho'}]\|_\infty = \|[i^*\hat\omega_\rho] - [i^*\hat\omega_{\rho'}]\|_\infty <\epsilon/2.\]

We would now like to argue that $\|[D_{\gamma}^*i^*\hat\omega_{\rho}] -[i^*\hat\omega_{\rho}]\|_\infty >\epsilon $.  This follows immediately from Theorem \ref{separation} in the case that $\gamma\in \pi_1(S)$ was already simple, because $D_{\gamma}(\lambda) \not = \lambda$.   
However $\gamma\in \pi_1(S)$ was not necessarily simple (which is why we are working in the cover $i:S'\to S$ where it is simple), so we need a more delicate argument.  Theorem \ref{separation} only supplies us with the bound $\|[D_{\gamma}^*i^*\hat\omega_{\rho}] -[i^*\hat\omega_{\rho}]\|_\infty >\epsilon(S')$, and $\epsilon(S')$ may go to zero as the genus $g'$ of $S'$ tends to infinity.  To circumvent this issue, we build a sequence of chains $\{V_j\}$ for $M_\rho$ satisfying $\displaystyle \frac{|\hat\omega_{\rho}(V_j)|}{\|V_j\|_1} > \epsilon - \frac{M}{\|H_j\|}$ for some constant $M$ as in the proof of Theorem \ref{separation}.  The finite sheeted covering $i: S'\to S$ induces a transfer map $\tau: C_\cdot (S\times I) \to C_\cdot (S'\times I)$, and we take $V_j' = \tau(V_j)$.  Notice that \[\displaystyle \frac{|\hat\omega_{\rho}(V_j')|}{\|V_j'\|_1} = \displaystyle \frac{\deg(i)|\hat\omega_{\rho}(V_j)|}{\deg(i)\|V_j\|_1} = \displaystyle \frac{|\hat\omega_{\rho}(V_j)|}{\|V_j\|_1} > \epsilon - \frac{M}{\|H_j\|}.\]  Let $\lambda'$ be the geodesic lamination on $S'$ satisfying $i(\lambda') = \lambda$.  We can argue as in the proof of Theorem \ref{mainthm1} that since $D_\gamma(\lambda')\not = \lambda'$, there is a constant $c>0$ such that $|D_\gamma^*i^*\omega_\rho(V_j')|<c$ for all $j$.  Thus given $e>0$, there is an $N$ such that for $j\ge N$
\begin{align*}
\frac{|(i^*\hat\omega_{\rho}-D_{\gamma}^*i^*\hat\omega_{\rho} +\delta B)(V_j')|}{\|V_j'\|_1} & \ge \epsilon - \frac{M}{\|H_j\|} - \frac{c + 32(g'-1)\|B\|_\infty}{\|V_j'\|_1} \\ & > \epsilon - e,
\end{align*}
where $B\in \Cb_b^2(S\times I)$ is chosen arbitrarily.  
We conclude that $\|[D_{\gamma}^*i^*\hat\omega_{\rho}] -[i^*\hat\omega_{\rho}]\|_\infty >\epsilon$ in general.  
Thus  \begin{align*}
\epsilon & < \|[D_{\gamma}^*i^*\hat\omega_{\rho}] - [i^*\hat\omega_{\rho}]\|_\infty\\ &
\le \|[D_{\gamma}^*i^*\hat\omega_{\rho}] - [D_{\gamma}^*i^*\hat\omega_{\rho'}]\|_\infty+ \|[i^*\hat\omega_{\rho'}] - [i^*\hat\omega_{\rho}]\|_\infty \\ 
& = 2 \|[i^*\hat\omega_\rho] - [i^*\hat\omega_{\rho'}]\|_\infty <2\epsilon/2 = \epsilon.
\end{align*}
This is a contradiction in the case that $S$ was closed, so $\rho'$ was faithful.  If $S$ is not closed, then $\bH^3/\im\rho$ is homeomorphic to a handlebody $H$ of genus equal to the rank of $\pi_1(S)$.  In this case, take $\epsilon'<\epsilon(\partial \overline H)/2$.  We need only to notice that if $\gamma\in \ker\rho'$ is non-trivial, then in the finite sheeted cover of $H'$ where $\gamma$ is simple on $\partial \overline{H'}$, $D_\gamma: \partial \overline{H'}\to \partial \overline{H'}$ does not extend to a homeomorphism of $H'$ that is homotopic to the identity on $H'$.  This means that again $D_\gamma(\lambda') \not= \lambda' \in \cEL(\partial \overline {H'}, H')$, and so the argument above goes through analogously where, $\partial \overline H$ plays the role of $S$.  
\end{proof}

\bibliography{groups}{}
\bibliographystyle{amsalpha.bst}
\end{document}